%% file: main.tex
\documentclass{article}%%%%where rsproca is the template name

\usepackage[a4paper, total={6in, 8in}]{geometry}
\usepackage{grffile}
\usepackage{graphicx}% Include figure files
\usepackage{algorithmicx,algpseudocode}
\usepackage[subfigure]{tocloft}
\usepackage{subfigure}
\usepackage{algorithm,algpseudocode}
\usepackage{comment}
\usepackage{amssymb}
\usepackage{amsmath}
\usepackage[english]{babel}
\usepackage{float}
\usepackage{accents}
\usepackage{bm}
\usepackage{amsthm}
\usepackage{amsthm,thmtools,xcolor} % Coloured "Theorem"
\usepackage{cases}
\usepackage{biblatex} %Imports biblatex package
\newcommand{\bb}[1]{\mathbf{#1}}

\newcommand{\x}{\bb{x}}

\newcommand{\R}{\mathbb{R}}
\newcommand{\nor}{\bb{n}}

\newcommand{\Ht}{L^{2}(0,T;H^1(\Omega_{un}))}

\newtheorem{theorem}{Theorem}
\newtheorem{remark}{Remark}
\newtheorem{prop}{Proposition}

\author{%%%% Author details
Riccardo Saporiti $^{1}$, Carlo Sinigaglia$^{2}$, Andrea Manzoni $^{3}$ and Francesco Braghin$^{2}$}

\title{An optimal control strategy to design passive thermal cloaks of arbitrary shape}

\begin{document}

\maketitle

\thanks{Preprint accepted for publication in a future issue of the Proceedings of the Royal Society A \\ \\}

%%%% Article title to be placed here
\thanks{$^{1}$ École Polytechnique Fédérale de Lausanne, Institute of Mathematics, 
Lausanne 1015, Switzerland \\
$^{2}$Politecnico di Milano, Department of Mechanical Engineering, Milano 20133, Italy\\
$^{3}$Politecnico di Milano, MOX - Department of Mathematics, Milano 20133, Italy\\}
%%%% Abstract text to be placed here %%%%%%%%%%%%
\begin{abstract}
\input{sections/abstract.tex}

\end{abstract}
%%%%%%%%%%%%%%%%%%%%%%%%%%%

%\footnote{Preprint accepted for publication in a future issue of the Proceedings of the Royal Society A.}

\section{Introduction}
\label{intro}
\input{sections/intro_passive.tex}

\section{The optimal control problem: formulation and analysis}
\label{OCP}
\input{sections/ocp.tex}

\section{Discretization and Numerical experiments}
\label{NUM}
\input{sections/num_sim.tex}

\section{Conclusion}
\label{conclusion}
\input{sections/conclusion.tex}

\appendix{\input{sections/appendix}}
\printbibliography

%%%%%%%%%% Insert bibliography here %%%%%%%%%%%%%%

\end{document}

%% file: sections/abstract.tex
%In this paper, we propose the possibility to achieve passive thermal cloaking using optimal control techniques to properly modify the diffusivity of the material. We model the problem as a bilinear heat equation, in which the control variable is added to the diffusion coefficient.

\noindent In this paper we describe a numerical framework for achieving passive thermal cloaking of arbitrary shapes in both static and transient regimes. The design strategy is cast as the solution of an optimal control problem (OCP) for the heat equation where the coefficients of the thermal diffusivity matrix take the role of control functions and the distance between the uncloaked and the cloaked field is minimized in a suitable observation domain. The control actions enter bilinearly in the heat equation, thus making the resulting OCP nonlinear, and its analysis nontrivial. We show that optimal diffusivity coefficients exist both for the static and the transient case; we derive a system of first-order necessary optimality conditions; finally, we carry out their numerical approximation using the Finite Element Method. A series of numerical test cases assess the capability of our strategy to tackle passive thermal cloaking of arbitrarily complex two-dimensional objects. % effectively.

%Nonlinear constraints, to keep the overall diffusivity matrix positive definite, must be imposed during the optimization to avoid numerical instabilities.

%Aiming to reproduce a reference field derived by the heat equation solved in a domain without the obstacle, we cast the problem into the framework of PDE-constrained optimization. By the theoretical analysis of the problem, we prove the existence of an optimal control variable and of the associated optimal state. 

%The Finite-Element-method and the Theta-method are employed to get the fully discrete algebraic problem and the numerical results are shown for different configurations of the target shape and of the control domain.  

%% file: sections/intro_passive.tex
\noindent Preventing an observer from locating an object is the purpose of the solution of a cloaking problem. To reach this goal, the properties of the cloak must be such that the surrounding field is locally indistinguishable from the one observed without the object itself. Being able to solve such a problem is of key importance in many fields like, e.g., 
%
%In this paper, we follow another route and reformulate the design phase such that the properties of the cloak are obtained as the solution of a partial differential equation (PDE)- constrained optimization problem, that is, an optimal control problem (OCP). The control functions are infinite-dimensional, space-varying fields of material properties that nullify the scattered wave.
%
%\noindent Being able to solve cloaking problems  is a subject studied in several %fields of mathematical physics. Research in 
optics, wave propagation phenomena \cite{cummer2007one,norris2008acoustic,chen2017broadband,zhang2008cloaking,greenleaf2008approximate}, heat conduction \cite{schittny2013experiments,guenneau2013fick,schittny2014invisibility} and convection \cite{leonhardt2013cloaking}. This paper focuses on a passive technique for thermal cloaking -- that is, the need of masking the effect of heat conduction and convection due to the presence of an object -- of two-dimensional arbitrarily shaped cloaks. The heat equation, as many other Partial Differential Equation (PDEs), is form-invariant under curvilinear coordinate transformations \cite{pendry2006controlling,leonhardt2006optical,kadic2015experiments, guenneau2010colours,schittny2013experiments}. In all the domains in which this property holds, transformation theories have been developed to solve the cloaking problem. However, this approach is very binding in terms of the shape of both the cloak and the object to hide. 
A successful attempt at deriving feasible material properties for passive thermal cloaks of ellipsoidal shape is presented in \cite{han2018full} while cloaking of arbitrary polygonal shapes has been achieved in \cite{Zhang:08} for electromagnetic problems. Polygonal shapes are divided into triangles in which a local coordinate transformation allows to map each triangle to a trapezoid thus leaving room for the object to hide. In the context of thermal problems, a composite coordinate transformation on subset of polygonal domains has been used to achieve thermal cloaking in arbitrary shapes, see \cite{Xu:18} and references therein.
The material properties of the cloak are then dictated by the resulting transformation which may generate steep gradients or even singularities. It is then up to the ingenuity of the designer to come up with a transformation which is realizable in practice, for example with the use of thermal metamaterials. On the other hand, the optimization based approach we propose allows direct control on the diffusivity tensor without relying on a coordinate transformation. This in turns allows engineering constraints to be easily handled while keeping an optimal tradeoff for the performance of the cloak. Achieving thermal cloaking through PDE-constrained optimization allows to ensure a remarkable flexibility regarding the shape of the obstacle to hide as well as how the cloaking sources are placed. An optimization formulation is combined with a reduced order model to improve performance in \cite{Sinigaglia_2022}. In particular, an active thermal cloak can be modelled by inserting a set of additional control variables into the state heat equation: these act as external sources or forces and do not change the intrinsic properties of the material. In this case, the physical realization of the cloak would require thermoelectric materials that, creating heat sources and sinks, manipulate the heat flux as desired. While passive cloaks do not require energy consumption, the geometry of the objects that can be properly hidden by a passive cloak is in general very simple. The application of bilinear or multiplicative control variables, on the other hand, allows obtaining these changes. Several bilinear Optimal Control Problems (OCPs) \cite{manzonioptimal} constrained by PDEs, feature bilinear control variables and have already been analyzed. Remarkable examples are provided by {\em(i)} the wave equation with a control variable affecting the diffusion coefficient \cite{9787068,glowinski2022bilinear,Clason_2020}, {\em(ii)} the search for the best advection field generated by boundary actuation in an a mean field model describing the motion of a particles swarm in terms of an advection-diffusion equation, as well as {\em(iii)}   %the problem of identifying a reaction coefficient (\textit{\cite{manzonioptimal}, Chapter 9.5}). 
the design of an acoustic cloak \cite{Cominelli_2022}, just to mention a few.The aim of this work is to achieve thermal cloaking by relying on a passive technique -- rather than an active one -- and extend the approach to arbitrarily shaped cloaks. The flexibility of the optimization formulation provides a framework where constraints from actual experimental engineering design can be satisfied. In the acoustic domain, a similar idea was presented in \cite{Cominelli_2022} where an optimization-based approach is integrated with the design of a sonic metamaterial. Passive thermal cloaks can also be designed experimentally using metamaterials \cite{schittny2013experiments,yang2021controlling} but can also be realized using layering or other techniques. The approach presented in this paper provides a distributed anisotropic diffusivity field where constraints defined by the specific experimental realization can be added, thus obtaining an optimal cloak which satisfies design guidelines and achieves the optimum design under such conditions.

The mathematical formulation of this problem involves an optimal control problem for the heat equation with the diffusivity coefficient playing the role of control function. This assumption allows to cast the problem into the framework of bilinear OCPs. 
In particular, the control variables, being supposed to modify the diffusivity inside the cloak, are added to the background diffusivity. 
 We show the performances of our device in both stationary and time-dependent conditions, and we emphasize the benefits and the limitations that arise in this problem with respect to the case of active cloaking. Achieving passive cloaking is indeed more complex than active cloaking due to the necessity to impose nonlinear constraints in the optimization problem, and to the nonlinear interaction between the state and the control, this latter affecting the state operator itself rather than its source or boundary data, as in the case of active cloaking. 
 This fact entails that the solution to the OCP is preferably found, numerically, through an iterative method rather than solving a saddle-point problem arising from the all-at-once imposition of the optimality conditions, as it would be the result in the case of active cloaking. Moreover,  
 the proof of the existence of an optimal control is also more difficult to obtain compared to the case of the linear-quadratic OCP arising from active cloaking; for this reason, an entire section of the paper is devoted to the analysis of the resulting OCP, before proceeding to its numerical approximation.  
   
The structure of this paper is the following. In Section 2 we set the (infinite-dimensional) OCP, derive first-order necessary optimality conditions, and proceed by discretizing the problem with the Finite Element Method (FEM). In Section 3 we derive a series of theoretical results required to analyze the well-posedness of the mathematical problem; in particular, we prove the existence and uniqueness of an optimal state for any given control variable, also showing that, under suitable assumptions on the functional spaces, we can obtain the Fréchet differentiability of the control-to-state map and the existence of an optimal control variable. In Section 4 we show the results obtained through our numerical simulations, both in the stationary case and in the time-dependent case.  Finally, some conclusions are drawn in Section 5.

\vspace{-0.1cm}

%% file: sections/ocp.tex
In this Section, we formulate the cloaking problem as an OCP with a state constraint given by a PDE. As a reference unperturbed state, we consider the thermal field generated by a source in a bounded domain $\Omega_{un}$ and the time interval $(0,T)$. The reference thus satisfies the heat equation: \vspace{-0.1cm}
\begin{equation}
\label{heat_basic}
    \rho c_{p} \frac{\partial z(\x,t)}{\partial t} - k \Delta z(\x,t) = \tilde{s}(\x) \quad \text{in} \quad \Omega_{un}\times (0,T) \vspace{-0.1cm}
\end{equation}
where the temperature $z$ is measured in Kelvin, $\x \in \R^2$ is the spatial coordinate in meters, $t$ indicates time in seconds, $c_{p}$ is the specific heat ($\text{J}\text{K}^{-1}\text{kg}^{-1}$), $\rho$ the mass density ($\text{kg}\,\text{m}^{-3}$) and $k$ is the thermal conductivity ($\text{W}\,\text{m}^{-1}\,\text{K}^{-1}$) of the material. It is useful to define the thermal diffusivity $\mu = \frac{k}{\rho c_p}$ ($\text{m}^2\text{s}^{-1}$) and rewrite Equation (\ref{heat_basic}) equivalently as: \vspace{-0.1cm}
\begin{equation}
\label{heat_basic}
     \frac{\partial z}{\partial t} - \mu \Delta z = s  \quad \text{in} \quad \Omega_{un}\times (0,T) \vspace{-0.1cm}
\end{equation}
where $s = \frac{\tilde{s}}{\rho c_p}$ ($\text{K}\text{s}^{-1}$) is the source term.
To solve it, we need to equip the problem with suitable initial and boundary conditions. For the sake of simplicity, we consider a square domain with homogeneous Robin boundary conditions that approximate to first-order an unbounded domain \cite{ABCs}. Hence, the unperturbed field $z \in \Ht$ is the solution of \vspace{-0.1cm}
\begin{equation}
\label{reference_state}
\begin{cases}
\begin{array}{ll}
\frac{\partial z}{\partial t} -\mu\Delta z    = s &  \textrm{in} \quad \Omega_{un}\times (0,T) 
\phantom{space} \\
-\mu\nabla z \cdot \nor + \alpha \, z = 0 & \textrm{on} \quad \partial \Omega_{un} \times (0,T)  \\
z(\x,0)= 0 & \textrm{in} \quad \Omega_{un}\times \{0\},  
\end{array}
\end{cases}
\end{equation}
whereas the steady-state reference is simply obtained by considering the Poisson equation $ -\mu\Delta z    = s$ with the same boundary conditions on $\Gamma_d = \partial \Omega_{un}$ and neglecting the initial conditions. 
To take into account far-field effects, still relying on a bounded domain, the simple choice $\alpha=1$ allows to obtain good results in practice \cite{ABCs}.  

%The unperturbed layout together with the steady-state solution obtained through the FEM for selected values of $\mu$,$\alpha$,$s$ is shown in Figure \ref{reference_layout_fig}.

\noindent We now consider the temperature field generated by considering the presence of an obstacle \linebreak $\Theta \subset \Omega_{un}$ whose temperature is kept constant. The obstacle thus induces nonhomogeneous boundary conditions of Dirichlet type along its boundary $\Gamma_o = \partial \Theta$ that perturbs the background thermal field. The thermal cloak domain surrounding the obstacle is denoted by $\Omega_c$ and it is the domain for the control functions $u,v,f$ which affects the diffusivity matrix. 
We denote as $q \in {L^{2}(0,T;H^1(\Omega))}$ the temperature field in the presence of the obstacle, that is the solution of \vspace{-0.1cm}
\begin{equation}
\label{state_problem}
\begin{cases}
\begin{array}{ll}
\frac{\partial {q}}{\partial t} -\nabla{}\cdot\left\{[\mu I + u\chi_{c}U+ f\chi_{c}L + v\chi_{c}L]\nabla{q}\right\} = s &  \textrm{in} \quad \Omega \times (0,T) 
\phantom{space} \\
-\mu\nabla q \cdot \nor + q = 0 & \textrm{on} \quad \Gamma_d \times (0,T) \\
q = T_o & \textrm{on} \quad \Gamma_o \times (0,T) \\
q(\x,0)= 0 & \textrm{in} \quad \Omega \times \{0\}.  
\end{array}
\end{cases}
\end{equation}
where $\Omega = \Omega_{un} \setminus \Theta $ and $\partial \Omega = \Gamma_o \cup \Gamma_d $.
$I$ is the identity matrix, while the control functions $u$, $f$, $v$ multiply respectively the matrices: \vspace{-0.1cm}
\begin{equation*}
U= \left( \begin{array}{ccc}
1 & 0\\
0 & 0\end{array} \right), \quad 
L= \left( \begin{array}{ccc}
0 & 0\\
0 & 1\end{array} \right), \quad 
S= \left( \begin{array}{ccc}
0 & 1\\
1 & 0\end{array} \right).
\end{equation*}
In this way, the total diffusivity matrix results in \vspace{-0.1cm}
\begin{equation} \label{diffusivity_matrix}
    K = \begin{pmatrix}
        \mu+u && v \\ 
        v && \mu + f
    \end{pmatrix}\, ;
\end{equation}
above, $\chi_{c}$ denotes the characteristic function of the control domain, that is, \vspace{-0.1cm}
\begin{equation*}
\chi_{c}(\textbf{x})=
\begin{cases}
    1,& \text{if } \textbf{x}\in\Omega_{c}\\
    0,              & \text{otherwise}.
\end{cases}
\end{equation*}

\begin{figure}[t!h]
\vspace{-0.2cm}
\centerline{
\includegraphics[width=0.5\textwidth]{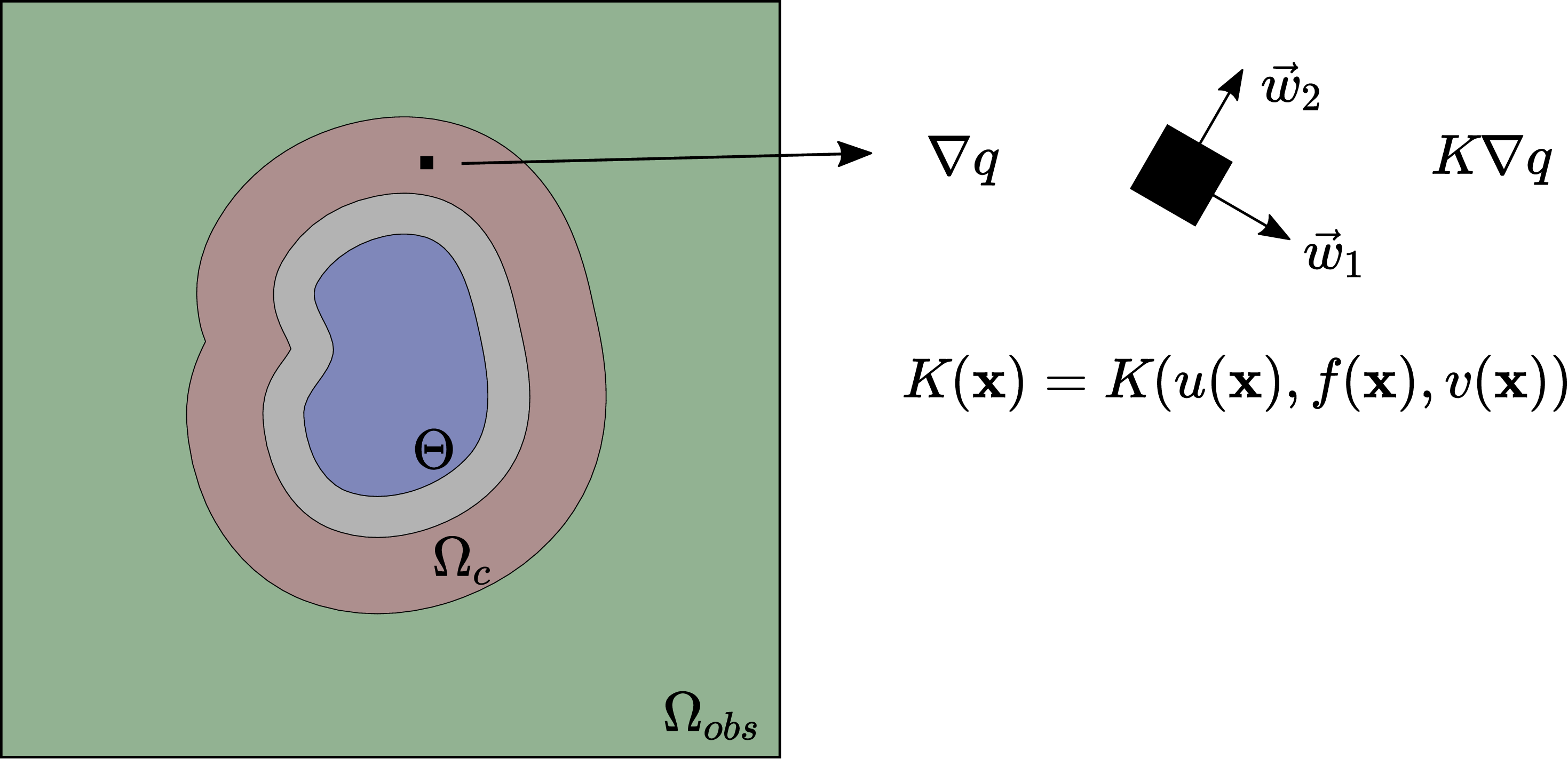}
}
	\caption{Conceptual domains definition and control strategy. The control triple $(u,f,v)$ modifies the diffusivity properties of the material inside the cloak, this latter being regarded as control domain $\Omega_c$. The optimal diffusivity matrix $K$ modifies the heat flow around the obstacle $\Theta$ to minimize the thermal trace in the observation domain $\Omega_{obs}$. The eigenvectors of the diffusivity matrix $\protect \overrightarrow{w}_1$, $\protect \overrightarrow{w}_2$ define the direction of maximum and minimum diffusivity, respectively.} %\rcc{per gli autovettori ho usato una la notazione con il comando "textbf", potremmo tenere quella anche in questa caption}
	\label{control_layout}
 \vspace{-0.45cm}
\end{figure}

We now turn to the setup of the Optimal Control Problem (OCP), referring to Figure \ref{control_layout} for its illustration. We consider an annular region $\Omega_c \subset \Omega$ in which a space-time varying control source $u(\x,t)$ is defined.
The cloaking objective can be rephrased as finding $u,f,v\in\mathcal{U}$ so that the state $q$ is as close as possible to the reference thermal field $z$ onto an observation region $\Omega_{obs} \subset \Omega$, exterior to the object.
%The functional space $\mathcal{U}$ to which the control variables belong will be defined accurately in the section devoted to the well-posedness of the problem.
In our case, the control region is defined as a circular annulus with thickness $r_c$ surrounding the target to cloak. Note that the flexibility allowed by the PDE-constrained formulation allows the control region to be arbitrarily shaped in the domain thus bridging the gap toward realistic applications.

Therefore, the cloaking problem can be formulated as the following OCP: \vspace{-0.1cm}
\begin{align} \label{ocp_formulation_functional}
\begin{split}
    \min\limits_{u,f,v,q} J(u,f,v,q) = &\frac{\alpha}{2}\int_{0}^{T}\int_{\Omega_{obs}} (q-z)^2 \,d\Omega dt + \\ &\int_{0}^{T}\int_{\Omega_{c}} \frac{\beta}{2} u^2 + \frac{\beta_{g}}{2} \|\nabla{u}\|^2 + \frac{\xi}{2} f^2 + \frac{\xi_{g}}{2} \|\nabla{f}\|^2 + \frac{\gamma}{2} v^2 + \frac{\gamma_{g}}{2} \|\nabla{v}\|^2 \,d\Omega dt, 
\end{split}
\end{align}
\begin{equation} \label{ocp_formulation_pde}
\textrm{subject to} 
    \begin{cases}
    \begin{array}{ll}
\displaystyle      \frac{\partial{q}}{\partial{t}}-\nabla{}\cdot\left\{[ \mu I + u\chi_{c}U + f\chi_{c}L + v\chi_{c}S ]\nabla{q}\right\}=s &\text{in $\Omega\times(0,T)$}
          \\
          q = T_{o}&\text{on $\Gamma_o\times(0,T)$}
          \\
          -\mu\nabla{q}\cdot \nor+\alpha q = 0&\text{on $\Gamma_{d}\times(0,T)$}
          \\
          q(\textbf{x},0) = 0&\text{in $\Omega\times(0,T)$},
        \end{array}
        \end{cases}
\end{equation}
where $z$ solves the reference problem (\ref{reference_state}). The OCP \eqref{ocp_formulation_functional}--\eqref{ocp_formulation_pde} consists of the minimization of a quadratic cost functional subject to the nonlinear constraint defined by the parabolic heat equation, with the nonlinearity resulting from the product between the control and the state variables. %\textcolor{red}{check $\alpha$: da un certo punto in poi dovrebbe essere 1, poi fa anche casino con il coefficiente nel funzionale costo. \textcolor{blue}{da qui in poi è fissato a 1} } 
For the sake of simplicity, and without losing generality, hereon we will take $\alpha=1$.

\begin{remark}
 The OCP above fits the theory of Bilinear Control problems. Indeed, the control terms directly  affect the diffusivity of the cloak. Several cases of multiplicative controls have already been treated, see,  e.g., \cite{doi:10.1137/17M1139953,https://doi.org/10.48550/arxiv.2002.03988,Clason_2020 }. %The theory behind nonlinear OCPs  constrained by PDEs is well developed in, e.g., \cite[Chapter 9]{manzonioptimal}.
While the analysis carried out in \cite{Sinigaglia_2022} involved a linear OCP in which the distributed control played the role of an additive source, now we turn to the problem of optimizing the material properties -- that is, acting on the state operator -- to get the desired cloaking. 
\end{remark}

Similarly, we can state the steady-state OCP as follows: 
\begin{align} \label{SS_ocp_functional}
\begin{split}
    &\min\limits_{u,f,v,q} J_{ss}(u^{ss},f^{ss},v^{ss},q^{ss})= \frac{1}{2}\int_{\Omega_{obs}} (q^{ss}-z^{ss})^2 \,d\Omega + \\
     &\int_{\Omega_{c}} \frac{\beta}{2} (u^{ss})^{2} + \frac{\beta_{g}}{2} \| \nabla{}u^{ss}\|^{2} + \frac{\xi}{2} (f^{ss})^{2} + \frac{\xi_{g}}{2} \|\nabla{}f^{ss}\|^{2} + \frac{\gamma}{2} (v^{ss})^{2} + \frac{\gamma_{g}}{2} \|\nabla{}v^{ss}\|^{2} \,d\Omega,
\end{split}
\end{align}

\begin{equation} \label{SS_pde_formulation}
    \textrm{subject to}  \begin{cases}
    \begin{array}{ll}
          -\nabla{}\cdot[(\mu I+  u^{ss}\chi_{c}U + f^{ss}\chi_{c}L  + v^{ss}\chi_{c}S)\nabla{}q^{ss}=s&\text{in $\Omega$}
          \\
          q^{ss} = T_{o}&\text{on $\Gamma_{o}$}
          \\
          -\mu\nabla{q^{ss}}\cdot \textbf{n}+ q^{ss} = 0&\text{on $\Gamma_{d}$}.
        \end{array}
        \end{cases}
\end{equation}

We can now derive a system of first-order necessary optimality conditions for the transient problem, exploiting the Lagrange method \cite{fredi,herzog}. To this aim, we can introduce the Lagrangian functional $\mathcal{L}: \mathcal{V} \times \mathcal{U}^3 \times \mathcal{W}^{*} \mapsto \R$ expressed by 
\begin{equation*}
    \mathcal{L}(q,u,f,v,p) = \tilde{J}(u,f,v,q) + \langle p , G(q,\bb{u}) \rangle_{\mathcal{W}^{\star},\mathcal{W}}
\end{equation*}
where, using the Lagrange multiplier $p\in  \mathcal{W}^{*} $, we "added" to the cost functional $J$ the nonlinear state equation expressed in abstract form by the operator $G:\mathcal{V} \times \mathcal{U}^3 \rightarrow \mathcal{W}$, whose precise definition will be given later.
The functional setting involved $\mathcal{V}=H^{1}(0,T;H^{1}(\Omega),H^{1}(\Omega)^{\star})$ and $\mathcal{W}=L^2(0,T;H^1_{\Gamma_{o}}(\Omega)^{*})$ so that state, control and adjoint variables are considered independently; for the well posedness of the problem, we require that $\mathcal{U}=\{u\in L^{\infty}(0,T;H^2(\Omega_{c})), \frac{\partial{u}}{\partial{t}}\in L^2(0,T;H^2(\Omega_{c})^{\star})\}$.
The set of first-order necessary optimality conditions can be therefore found by imposing that the G\^{a}teaux derivatives of the Lagrangian with respect to $(q,u,f,v,p)$, along an arbitrary variation $(\delta q,\delta u,\delta f, \delta v,\delta p)$, vanish. 
In our case, the Lagrangian takes the explicit form
\begin{equation}
\label{lag_def}
    \mathcal{L} = J(q,u,f,v) + \int_{\Omega}\int_0^T \left( -\frac{\partial q}{\partial t} + \nabla{}\cdot\{[\mu I + u\chi_{c}U + f\chi_{c}L + v\chi_{c}S ]\nabla{q}\} + s \right) p  \,\,  d td \Omega ,
\end{equation}
and the resulting optimality system in strong form consists of the state problem (\ref{state_problem}), the adjoint problem
\begin{equation}
\label{adjoint_problem}
\begin{cases}
\begin{array}{ll}
\displaystyle -\frac{\partial p}{\partial t} -\nabla{}\cdot\{ [ \mu I  + u\chi_{c}U + f\chi_{c}L + v\chi_{c}S ] \nabla{p} \}    = (q-z)\chi_{obs} &  \textrm{in} \quad \Omega \times (0,T) 
\phantom{space} \\
-\mu\nabla p \cdot \nor + p = 0 & \textrm{on} \quad \Gamma_d \times (0,T) \\
p = 0 & \textrm{on} \quad \Gamma_o \times (0,T) \\
p(\x,T)= 0 & \textrm{in} \quad \Omega \times \{T\}, \\
\end{array}
\end{cases}
\end{equation}
and the following optimality conditions (Euler equation):
\begin{equation} \label{euler_equation}
    \begin{cases}
    \begin{array}{ll}
        ( \beta u - (U\nabla{q})\cdot\nabla{p}, \delta u )_{{L^2(Q_{c_{T}})} } + ( \beta_{g} \nabla{}u , \nabla{}(\delta u) )_{{L^2(Q_{c_{T}})} } = 0&\forall h \in \mathcal{U} \\
              ( \xi f - (L\nabla{q})\cdot\nabla{p} , \delta f )_{{L^2(Q_{c_{T}})} } + ( \xi_{g} \nabla{}f , \nabla{} (\delta f) )_{{L^2(Q_{c_{T}})} } = 0&\forall r \in \mathcal{U} \\
              ( \gamma v - (S\nabla{q})\cdot\nabla{p} , \delta v )_{{L^2(Q_{c_{T}})} } + ( \gamma_{g} \nabla{} v , \nabla{} (\delta v) )_{{L^2(Q_{c_{T}})} } = 0&\forall k \in \mathcal{U}, 
    \end{array}
    \end{cases}
\end{equation}
where $Q_{c_{T}} = \Omega_{c}\times (0,T)$ and 
\begin{equation*}
    (y,\psi)_{L^2(Q_{c_{T}})} = \int_{0}^{T} (y(t),\psi(t))_{L^2(\Omega_{c})}dt.
\end{equation*}

 In the adjoint problem (\ref{adjoint_problem}), $\chi_{obs}$ denotes the characteristic function of the observation domain $\Omega_{obs}$.
 Since this kind of OCPs is rather standard, we skip the details about the derivation of the optimality system and refer the reader to \cite[Chapter 3]{fredi}.

\subsection{Well-posedness analysis of the stationary Optimal Control Problem}

This section is devoted to prove the existence of at least one optimal control for both the static problem.
 The proof of the main theorems is provided in the main text, while the proofs of theorems on well-posedness of the state problems and differentiability of the control-to-state maps are provided in Appendix \ref{appendix_proof}.

The analysis of the problem is carried out considering a unique isotropic control $u$ acting through an identity matrix $I$. Here $\Omega\subset\mathbb{R}^2$ is a bounded domain with $\partial{\Omega}$ Lipschitz. The extension to the case with three control terms is immediate since we are working in linear vector spaces; for this reason, we consider the following minimization problem, resulting from a slight simplification of the state problem \eqref{SS_pde_formulation}: %\textcolor{red}{check: il termine in $u$ non dovrebbe anche entrare nella condizione di Robin in 2.14?} \textcolor{blue}{dovrebbe andare bene anche così perchè la funcion $\chi_c$ è uguale a zero sul bordo di robin }
\begin{equation} \label{functional_analysis}
    \min\limits_{u,q} J(u,q) = \frac{1}{2}\int_{\Omega_{obs}} (q-z)^2 \,d\Omega + \frac{\beta}{2} \int_{\Omega_{c}} u^2 \,d\Omega
\end{equation}
    \begin{equation}\label{pde_analysis}
    \textrm{subject to}
    \begin{cases}
    \begin{array}{ll}
    \nabla{}\cdot\{-[(\mu+u\chi_{c}(\textbf{x}))I ]{\nabla{q(\textbf{x})}}\}=s(\textbf{x} )&\text{in $\Omega$}
          \\
          q(\textbf{x}) = T_{0}&\text{on $\Gamma_{o}$}
          \\
          -\mu\nabla{q}\cdot \textbf{n}+ q(\textbf{x}) = 0&\text{on $\partial\Omega_{un}=\Gamma_{d}$}.
    \end{array}
    \end{cases}
\end{equation}
Let $R:H^{1/2}(\Gamma_{o})\rightarrow H^1(\Omega)$ be the $lifting$ $of$ $the$ $boundary$ $data$ (see, e.g., \cite[Chapter 3.3.3]{quart}), such that 
\begin{equation*}
    R(g)|_{\Gamma_{o}} = g \quad \text{for all $g\in H^{1/2}(\Gamma_{o})$}.
\end{equation*}
We suppose that such lifting is a linear and continuous operator and we denote $\tilde{q}=q-R(T_0)$. 
Notice that $\tilde{q}|_{\Gamma_{o}} = q|_{\Gamma_{o}} - R(T_0)|_{\Gamma_{o}} = 0$ and $\nabla{\tilde{q}}=\nabla{q} - \nabla{R(T_0)}$. 
The weak form of problem \eqref{pde_analysis} therefore reads as:  
\begin{equation} \label{weak_fomulation_analysis}
    \text{find $\tilde{q}\in V$ \ : \ \ $a(\tilde{q},v)$} = Fv \quad \forall v\in V
\end{equation}
where $V = H^{1}_{\Gamma_{o}}(\Omega)=\{f\in H^{1}(\Omega):\:\text{trace}_{\Gamma_{o}}(f)=0\}$, 
\begin{equation} \label{bilinear_form_and_linear_functional}
    \text{$a(q,p)$} = \int_{\Omega} (\mu + u\chi_{c})I\nabla{q}\cdot\nabla{p} \,d\Omega - \int_{\Gamma_{d}} qp \,d\Gamma,\quad Fv = \tilde{F}v- \text{$a(R (T_0),v)$},\quad \tilde{F}v= \int_{\Omega} sv d\Omega.
\end{equation}

Once \eqref{weak_fomulation_analysis} is solved, the solution to the state problem can be retrieved as  $q  = \tilde{q} +  R (T_0)$.

Moreover, we define the isotropic diffusivity matrix: $K(\textbf{x}) = ( \mu + u \chi_{c}(\textbf{x}))I$ and recall the definition of $H^1(\Omega)$-norm:
\begin{equation}
    \|f\|_{H^1(\Omega)} = \|f\|_{V} = \sqrt{\int_{\Omega} |f|^2 + |\nabla{f}|^2 \,d\Omega} \qquad \forall f\in H^{1}(\Omega). 
\end{equation}
Analogously, the $H^2(\Omega_{c})$-norm is 
\begin{equation*}
    \|f\|_{H^2(\Omega_{c})}^2 = \|f\|_{\mathcal{U}}^2 = \sum_{|\alpha|\leq 2} \int_{\Omega_{c}} |D^{\alpha}f|^2 \,d\Omega,\quad\forall f\in H^2(\Omega_{c}). 
\end{equation*}

We introduce the differential operator $\mathcal{E}$ as 
\begin{equation*}
    \mathcal{E}q =  \nabla{}\cdot\{-[(\mu+u\chi_{c}(\textbf{x}))I ]{\nabla{q(\textbf{x})}}\}.
\end{equation*}
To ensure that the integrals in the analytic expression of $a$ are well-defined we will assume that the operator $\mathcal{E}$ is $uniformly$ $elliptic$, i.e. there exist positive numbers $\hat{\alpha}$ and $M$ such that
\begin{equation*} % \label{uniformly_elliptic}
    \hat{\alpha} \lvert \pmb{\xi} \rvert ^ {2} \leq K(\pmb{x})\pmb{\xi}\cdot \pmb{\xi} \leq M  \lvert \pmb{\xi} \rvert ^ {2} \quad \forall \pmb{\xi}\in\mathbb{R}^n \quad \text{a.e. in $\Omega$}.
\end{equation*}
By the uniform ellipticity condition, $K$ is $positive$ $definite$ in $\Omega$. Since $K$ is symmetric, its eigenvalues are real and the minimum eigenvalue is bounded from below by $\hat{\alpha}$, similarly its maximum eigenvalue is bounded from above by $M$.
 As a preliminary result, we start by proving the well-posedness of problem \eqref{pde_analysis} for every control $u \in \mathcal{U}$ (see Appendix \ref{proofTHM1} for the proof):
\begin{theorem}[Well posedness of the state problem \eqref{pde_analysis}] \label{Theorem_1}\ \\
Assume that $u\in\mathcal{U}$, with $\mathcal{U}=H^{2}(\Omega_{c})$, is such that the operator $\mathcal{E}$ is uniformly elliptic and that the bilinear form $a(\cdot,\cdot)$ defined in \eqref{bilinear_form_and_linear_functional} is coercive.
Then, for every  $s\in\mathbb{R}$, there exists a unique weak solution $q\in H^{1}(\Omega)$ to problem \eqref{pde_analysis}. Furthermore, the following a-priori-bound holds:
\begin{equation} \label{lax-milgram bound}
    \|\tilde{q}\|_{V}\leq \frac{|s| + \tilde{M}\|R T_0\|_{V}}{\tilde{\alpha}},
\end{equation}
where $\tilde{\alpha}$ and $\tilde{M}$ are, respectively, the coercivity and continuity constants of the bilinear form. 
%\begin{proof}
 %   see the Appendix \ref{proofTHM1}.
%\end{proof}
\end{theorem}

\vspace{-0.2cm}
%Hereon, for the sake of notation, we will denote the $\tilde{q}$ by $q$. 

In order to prove existence of optimal controls and derive the associated system of first-order necessary conditions for optimality we need to show that the control-to-state map $\Theta:\mathcal{U}\rightarrow V$, $q=\Theta(u)$, is Fr\'echet differentiable, i.e., that for every $u\in \mathcal{U}$ there exists an operator $\Theta'\in\mathcal{L}(\mathcal{U},\mathcal{L}(\mathcal{U},V))$ such that, for all $h\in \mathcal{U}$ for which $u+h\in \mathcal{U}$, the following relation holds:
\begin{equation*}
    \Theta(u+h) - \Theta(u) = \Theta'(u)h + R(u,h),
\end{equation*}
where $R(u,h)=o(\|h\|_{\mathcal{U}})$ for $\|h\|_{\mathcal{U}}\rightarrow 0$. Indeed, we can state the following (see Appendix \ref{proofTHM2} for the proof) result: \vspace{-0.2cm}
\begin{theorem}[Differentiability of the control to state map] \label{Theorem_2}\ \\
    The {\em control-to-state} map: $\Theta: \mathcal{U}\rightarrow V$ is Fr\'echet differentiable and the directional derivative $z = \Theta'(u)h$ at $u\in \mathcal{U}$ in the direction $h\in\mathcal{U}$ is the solution of \vspace{-0.1cm}
\begin{equation*}
 \begin{cases}
 \begin{array}{ll}
        - \nabla{}\cdot[  (\mu+u\chi_{c})\nabla{z}] = \nabla{}\cdot[h\chi_{c}\nabla{q}(u)] & \text{in $\Omega$} \\
        z = 0 &  \text{on $\Gamma_{o}$}\\
        -\mu\nabla{z}\cdot \textbf{n} + z = 0 & \text{on $\Gamma_{d}$}.
    \end{array}
    \end{cases}
\end{equation*}
\end{theorem}

We can now show that for the OCP \eqref{functional_analysis}-\eqref{pde_analysis} there exists at least an optimal control. 
For the proof, we consider homogeneous Dirichlet boundary conditions, that is, $T_{o}=0$. This assumption can be easily relaxed employing the lifting operator.  \vspace{-0.2cm}

\begin{theorem}[Existence of an optimal control]\ \\ 
    Suppose that the assumptions of  {\em Theorem} \ref{Theorem_1} hold, 
    then there exists at least an optimal solution $(\hat{q},\hat{u})\in V \times \mathcal{U}$ of the optimal control problem \eqref{functional_analysis}- \eqref{pde_analysis}.
    
\end{theorem}

\begin{proof}
    {\em Theorem} \ref{Theorem_1} already ensures that, for all fixed $u\in \mathcal{U}$, there exists a unique state $q=q(u)$. Thanks to {\em Theorem} \ref{Theorem_2}, the control-to-state map is $\mathcal{F}$-differentiable. 
    Then, let us rewrite the state equation by introducing the operator  $G: V \times \mathcal{U} \rightarrow V^{\star}$, given by
\begin{equation*}
    G(q,u) = Aq + B(u,q) - s,
\end{equation*}
where %\textcolor{red}{dove vanno a finire le condizioni al contorno su $\Gamma_0$ in questa formulazione?} \textcolor{blue}{ho aggiunto una riga per dire che considero $T_{o}=0$, la stessa proof funziona similmente}
\begin{equation*}
    a(q,\phi) = \langle Aq + B(u,q) , \phi \rangle 
\end{equation*}
with $A:V\rightarrow{}V^{\star}$ and  $B:\mathcal{U}\times V\rightarrow{}V^{\star}$ given by 
\begin{equation*}
     \langle Aq,\phi\rangle = \mu\int_{\Omega} \nabla{\phi} \cdot \nabla{q} \,d\Omega - \int_{\Gamma_{d}} q\phi \,d\Gamma, \qquad   
     \langle B(u,q),\phi\rangle = \int_{\Omega_{c}} u\nabla{q} \cdot \nabla{\phi} \,d\Omega;
\end{equation*}
finally,
\begin{equation*}
    \langle s,\phi\rangle = \int_{\Omega} s\phi \,d\Omega.
\end{equation*}
The state equation therefore reads as follows:  
\begin{equation*}
        \langle G(q,u),\phi\rangle =\langle Aq + B(u,q) - s , \phi \rangle  = 0 \qquad \forall \phi\in V.
\end{equation*}
The existence of an optimal solution 
$(\hat{q},\hat{u})\in V \times \mathcal{U}$ of the OCP  \eqref{functional_analysis}- \eqref{pde_analysis} is ensured once the following assumptions are verified (see, e.g., \cite[Theorem 9.4]{manzonioptimal}):
\begin{enumerate}
    \item $\inf\limits_{(q,u)} J(q,u)= \zeta > -\infty$ and the set of feasible points is nonempty;
    \item a minimizing sequence $\{q_{k},u_{k}\}$ is bounded in $V\times\mathcal{U}$;
    \item the set of feasible points is weakly sequentially closed in $V\times\mathcal{U}$;
    \item $J$ is sequentially weakly lower semicontinuous. 
\end{enumerate}
\vspace{-0.15cm}

Indeed, considering hereon all the limits taken for $n \rightarrow \infty$, we have that: \vspace{-0.1cm}
    \begin{enumerate}
        \item the cost functional is bounded from below, since
        \begin{equation}
        \inf\limits_{(q,u)} J(q,u)= \zeta > -\infty,
    \end{equation}
    and the set of feasible points is non-empty since $J(q,u) \geq 0$;
       \item a minimizing sequence: $\{(q_{k},u_{k})\}_{k\in \mathbb{N}}\subseteq V\times \mathcal{U}$ such that $J(q_{k},u_{k}) \rightarrow \zeta$ and 
    $G(q_{k},u_{k})=0$ %\quad k\in\mathbb{N}$
    is bounded. This follows from the definition of $\mathcal{U}$ and the bound \eqref{lax-milgram bound}; 
    \item to prove that the set of feasible points is {\em weakly} sequentially closed in $V\times \mathcal{U}$, we consider the following sequences:
    \begin{equation*}
    u_{n} \rightharpoonup u \quad \textrm{$in \ $  $\mathcal{U}$},\quad q_{n} \rightharpoonup q \quad \textrm{$in$ $V$},\quad Aq_{n}  \overset{\star}{\rightharpoonup} \Lambda \quad in\: V^\star,\quad B(u_{n},q_{n})  \overset{\star}{\rightharpoonup} \Phi \quad in\: V^\star. 
    \end{equation*}
    We have that $(q,u)\in V\times \mathcal{U}$ since $V$ and $\mathcal{U}$ are {\em weakly} sequentially closed.
    Using this definition we, first of all, prove that $\Lambda = Aq$, showing that: 
    \begin{equation*} \label{1weak1}
        \langle Aq_{n} - Aq , \phi \rangle  \rightarrow 0 \quad \forall\phi\in V.
    \end{equation*}
    This amounts to show that
    \begin{equation*}
        \mu\int_{\Omega}( \nabla{q}_{n}-\nabla{q})\cdot\nabla{\phi}\,d\Omega - \int_{\Gamma_{d}} (q_{n}-q)\phi \,d\Gamma \rightarrow 0.
    \end{equation*}
    We consider in this step the equivalent definition of scalar product in $V=H^1_{\Gamma_{o}}(\Omega)$:
    \begin{equation*}
        (u,v)_{V} = \int_{\Omega} \nabla{u}\cdot\nabla{v} \,d\Omega 
    \end{equation*}
   so that we can exploit the $weak$ convergence in the reflexive space $V$. In particular, %we have: 
    \begin{equation*}
        q_{n} \rightharpoonup q \quad \text{in $V$} \iff  \nabla{q}_{n} \rightharpoonup \nabla{q} \quad \text{in $L^2(\Omega)$}.
    \end{equation*}
    Furthermore, by the continuity of the trace operator (see e.g.  \cite[Appendix A]{manzonioptimal}),  
    $q_{n}\rightharpoonup q$ in $L^2(\Gamma_{d})$. 
    We now prove that 
    \begin{equation}
        \langle B(u_{n},q_{n}),\phi\rangle  \rightarrow  \langle B(u,q),\phi\rangle   \qquad \forall \phi\in V, 
    \end{equation}
    that is,
    \begin{equation} \label{convergence_u_n_u}
        \int_{\Omega_{c}} (u_{n}\nabla{q_{n}}-u\nabla{q})\cdot \nabla{\phi} \,d\Omega   \rightarrow 0  \qquad \forall \phi\in V
    \end{equation}
    and so $B(u,q)=\Phi$. The relation
    \eqref{convergence_u_n_u} is equivalent to  
    \begin{equation} \label{convergence_u_n_u_expanded}
        \int_{\Omega_{c}} \nabla{q}_{n}\cdot\nabla{\phi}(u_{n}-u) \,d\Omega + \int_{\Omega_{c}} ( \nabla{q_{n}} - \nabla{q} ) \cdot \nabla{\phi}u \,d\Omega  \rightarrow 0.
    \end{equation} 
Regarding the first term  of \eqref{convergence_u_n_u_expanded},
    \begin{equation*}
        \lim_{n\to \infty} \int_{\Omega_{c}} ( u - u_{n} ) \nabla{q}_{n} \cdot \nabla{\phi} \,d\Omega = 0 
    \end{equation*}
indeed, using the compact embedding $H^2(\Omega_{c})\hookrightarrow \hookrightarrow  L^{\infty}(\Omega_{c})$ valid in dimension $n=2,3$,
    \begin{equation*}
        \left|\int_{\Omega_{c}} (u-u_{n})\nabla{q}_{n}\cdot{\nabla{\phi}} \,d\Omega \right| \leq \|u-u_{n}\|_{L^{\infty}(\Omega_{c})}\|\nabla{\phi}\|_{L^2(\Omega)}\|\nabla{q}_{n}\|_{L^2(\Omega)} \rightarrow 0
    \end{equation*}
    where $\|u-u_{n}\|_{L^{\infty}(\Omega_{c})}\rightarrow 0$ and $\|\nabla{q}_{n}\|_{L^2(\Omega)}\leq M_{weak}$, with $M_{weak}>0$ since every weak convergent sequence is bounded.    
    For the second term in \eqref{convergence_u_n_u_expanded} we have that \vspace{-0.1cm}
    \begin{equation} \label{weakconv1}
        \lim_{n\to \infty}  \left|\int_{\Omega_{c}} ( \nabla{q_{n}} - \nabla{q} ) \cdot \nabla{\phi}u \,d\Omega\right| \rightarrow 0  \vspace{-0.1cm}
    \end{equation}
    where we have used again the fact that $\nabla{q_{n}} \rightharpoonup \nabla{q}$ in $L^2(\Omega)$ and that  
    $\nabla{\phi}u = \psi \in L^2(\Omega)$ since $\phi \in H^1_{\Gamma_{d}}(\Omega)$ and $u\in \mathcal{U}\subseteq L^{\infty}(\Omega_{c})$.   
    Gathering all these information together, we can conclude that $G(q,u)=0$ in $V^\star$ and so by the arbitrariness of the sequence, we have that the set of feasible points is {\em weakly} sequentially closed;
    
    \item Since $J(q,u)$ is convex and continuous in $V\times\mathcal{U}$, it is also lower semi-continuous with respect to the weak convergence: \vspace{-0.15cm}
    \begin{equation*}
        J(q,u) \leq \liminf\limits_{n\rightarrow \infty} J(q_{n},u_{n}) = \zeta. \vspace{-0.15cm}
    \end{equation*}
    \end{enumerate}
    Therefore, there exists at least an optimal solution $(\hat{q},\hat{u}) \in V \times \mathcal{U}$ of the OCP \eqref{functional_analysis}- \eqref{pde_analysis}.
\end{proof}

\subsection{Well-posedness analysis of the transient Optimal Control Problem }

We now turn to the analysis of the time-dependent problem, where we assume to be able to change the diffusivity of the material inside the cloak in both space and time. We start by formulating the time-dependent problem and prove that is well-posed. The time-dependent OCP can be written as \vspace{-0.1cm}
\begin{equation} \label{functional_td_analysis}
    \min\limits_{u(t),q(t)} J(u(t),q(t)) = \frac{1}{2}\int_{0}^{T}\int_{\Omega_{obs}} (q(t)-z(t))^2 \,d\Omega dt+ \frac{\beta}{2} \int_{0}^{T}\int_{\Omega_{c}} u(t)^2 \,d\Omega dt \vspace{-0.1cm}
\end{equation}
\begin{equation} \label{pde_td_analysis}
    \textrm{subject to} 
    \begin{cases}
    \begin{array}{ll}
    \displaystyle \frac{\partial{q}}{\partial{t}}(\textbf{x},t) + \nabla{}\cdot\{-[(\mu+  
    u(\textbf{x},t)\chi_{c}(\textbf{x}))I ]{\nabla{q(\textbf{x},t)}}\}=s(\textbf{x} )&\text{in $\Omega\times(0,T)$}
          \\
          q(\textbf{x},t) = 0&\text{on $\Gamma_{o}\times(0,T)$}
          \\
          -\mu\nabla{q}\cdot \textbf{n}+ q(\textbf{x}) = 0&\text{on $\Gamma_{d}\times(0,T)$} \\
          q(\textbf{x},0) = q_{0} = 0 &\text{in $\Omega\times\{0\}$}. \vspace{-0.1cm}
    \end{array}
    \end{cases}
\end{equation}
We now select suitable functional spaces to deal with the evolution problem involving the state dynamics, see, e.g.,  \cite[Chapter 7]{evans} or \cite[Chapter 7]{manzonioptimal} for further details. State and control spaces are chosen as  
$Y=H^1(0,T;V,V^{\star})$ and $\mathcal{U} = \{ u \in L^{\infty}(0,T;H^2(\Omega_{c})), \frac{\partial{u}}{\partial{t}}\in L^2(0,T;H^2(\Omega_{c})^{\star})\}$;
$\mathcal{U}$ is a Banach space with the norm \vspace{-0.15cm}
\begin{equation*}
    \|u\|_{\mathcal{U}} = \left( \|u\|_{L^{\infty}(0,T;H^2(\Omega_{c}))}^2 + \|\frac{\partial{u}}{\partial{t}}\|_{L^2(0,T;H^2(\Omega_{c})^{\star})}^2 \right)^{1/2};
\end{equation*}
this can be proven similarly to what was done in \cite[Theorem 11.4]{Chipot_2000} using the embedding  \vspace{-0.1cm}
\begin{equation*}
     L^{\infty}(0,T;H^{2}(\Omega_{c}))\hookrightarrow \mathcal{D}'(0,T;H^{2}(\Omega_{c})) \hookrightarrow \mathcal{D}'(0,T;H^{2}(\Omega_{c})^{\star}). \vspace{-0.1cm}
\end{equation*}
We consider the Hilbert triplet $V\subset H\subset V^{\star}$, with $H=L^2(\Omega)$, and define the parabolic operator  \vspace{-0.1cm}
\begin{equation*}
    \mathcal{P}q = \frac{\partial{q}}{\partial{t}} + \nabla{}\cdot\{-[(\mu+  
    u\chi_{c}(\textbf{x}))I ]{\nabla{q}}\}. \vspace{-0.1cm}
\end{equation*}
The variational formulation of \eqref{pde_td_analysis} then reads as follows: 
\begin{equation} \label{weak_td_problem}
          \int_{0}^{T} \langle\frac{\partial{q}}{\partial{t}}(t),\phi(t)\rangle dt + \int_{0}^{T} a(q(t),\phi(t))dt =  \int_{0}^{T} \langle s(t),\phi(t) \rangle dt \quad \forall \phi\in L^{2}(0,T;V),
\end{equation}

\begin{comment}
\begin{equation} \label{weak_td_problem}
\begin{cases}
\begin{array}{ll}
    \langle\frac{\partial{q}}{\partial{t}}(t),\phi\rangle  + a(q(t),\phi)=\langle s,\phi\rangle  & \text{$\forall \phi\in V$, in $\mathcal{D}'[0,T)$} \\
    q(0) = q_{0}\in H, 
\end{array}
\end{cases}
\end{equation}
meaning that
\begin{equation} \label{extended_weak_formulation}
    \begin{array}{l}
       \displaystyle   -\int_{0}^{T} (q(t),\phi)_{H}\psi'(t)dt - (q(0),\phi)_{H}\psi(0) + \int_{0}^{T} a(q(t),\phi)\psi(t)dt = \\
    \hspace{3.5cm} \displaystyle \int_{0}^{T} \langle s,\phi \rangle \psi(t)dt \quad \forall \phi\in V \quad \forall \psi(t)\in \mathcal{D}[0,T),
\end{array}
\end{equation}
\end{comment}

where the bilinear form $a(\cdot,\cdot)$ has already been defined in \eqref{bilinear_form_and_linear_functional}. 
\\
We assume that the operator $\mathcal{P}$ is uniformly parabolic: there exists $\alpha,M>0$: \vspace{-0.1cm}
\begin{equation*} 
    \alpha\lvert\textbf{y}\rvert^2 \leq \sum_{i,j=1}^{d}k_{i,j}(\textbf{x},t)y_{i}y_{j},\quad 
    \lvert k_{i,j}(\textbf{x},t) \rvert \leq M, \quad \forall \textbf{y}\in\mathbb{R}^d,\quad \text{a.e. in $Q_{c}^{T} = \Omega_{c}\times (0,T)$}, \vspace{-0.1cm}
\end{equation*}
where $k_{i,j}$ is the generic term of the matrix $K = ( \mu + u\chi_{c} )I  $.
Under this hypothesis, following the derivation of  {\em Theorem} \ref{Theorem_1}, we can show that the bilinear form $a$ is continuous and coercive.
Then, referring to \cite[Chapter 7.4.1]{manzonioptimal}, we state the following.  \vspace{-0.1cm}
\begin{prop}[Well posedness of the state problem \eqref{weak_td_problem}]\label{Prop_1}\ \\
    Let $s\in L^{2}(0,T;V^{\star})$, $q_{0}\in H$. Suppose that the operator  $\mathcal{P}$ is uniformly parabolic,
Then there exists a unique $q\in H^1(0,T;V,V^{\star})\subseteq C^{0}([0,T];H)$ weak solution of the abstract equation \eqref{weak_td_problem}. 
\end{prop}

\vspace{-0.1cm}

Starting from \eqref{weak_td_problem}, we can derive the following estimates on the norm of $q$:
\begin{equation*}
   \|q\|_{L^2(0,T;V)}^2 \leq K_{2}( s , q(0) , \tilde{\alpha}, T ), \label{L2_V}
   \end{equation*}
   \begin{equation} 
   \|q\|_{Y}^2 \leq  K_{2}(s,q(0),\tilde{\alpha},T) + 2\{ \tilde{M}^2[K_{2}(s,q(0),\tilde{\alpha},T)]+|s|^2 T \}\label{Y_norm},
\end{equation}
where
\[
K(s,q(0),\tilde{\alpha},T)= \frac{|s|^2T}{\tilde{\alpha}} + \|q(0)\|_{L^2(\Omega)}^2>0,
\] 
\[
K_{2}( s , q(0) , \tilde{\alpha} , T)= \frac{1}{\tilde{\alpha}} \left( |s|K(s,q(0),\tilde{\alpha},T)T + \frac{1}{2} \|q(0)\|_{L^2(\Omega)}^2 \right),
\] 
being $\tilde{M}$ and $\tilde{\alpha}$ the continuity and the coercivity constants of the form  \eqref{bilinear_form_and_linear_functional}, respectively. Using similar arguments to the %as in the proof of the differentiability of the control-to-state map in the 
stationary case, we can prove the differentiability of the control-to-state map for the time-dependent case: \vspace{-0.2cm} %. We have the following

\begin{theorem}[Differentiability of the control to state map] \label{Theorem_4}\ \\
The {\em control-to-state} map: $\Theta: \mathcal{U}\rightarrow Y$ is Fr\'echet differentiable and the directional derivative: $z = \Theta'(u)h\in Y$ at $u\in \mathcal{U}$ in the direction $h\in\mathcal{U}$ is given by the solution of 
\begin{equation}\label{pde_z_td_analysis}
 \begin{cases}
 \begin{array}{ll}
        \frac{\partial{z}}{\partial{t}} + \nabla{}\cdot[ - (\mu+u\chi_{c})\nabla{z}] = \nabla{}\cdot[h\chi_{c}\nabla{q}(u)] & in \:\:\Omega \times (0,T) \\
        z = 0 & on \:\: \Gamma_{o}\times(0,T)\\
        -\mu\nabla{z}\cdot \textbf{n} + z = 0 & on \:\: \Gamma_{d}\times(0,T)\\
        z(\textbf{x},0) = 0 & in \:\: \Omega\times \{0\}.
    \end{array}
    \end{cases}
\end{equation}
\begin{proof}
    see the Appendix \ref{proofTHM4}.
\end{proof}
\end{theorem}

To prove the existence of an optimal pair $(\hat{q},\hat{u})\in Y\times\mathcal{U}$, the state equation is recast in the dual space $V^\star$ for a.e. $t$ introducing the operator $G:V\times \mathcal{U} \rightarrow V^\star$ given by
\begin{equation*}
    G(q,u) = \begin{pmatrix} \dot{q}(t) + Aq(t) + B(u(t),q(t)) - s \\
    q(0) = q_{0}
    \end{pmatrix}.
\end{equation*} 
The bilinear form $a(q(t),\phi)$ can be split so that
\begin{equation*}
    a(q(t),\phi) = \langle Aq(t) + B(u(t),q(t)) , \phi \rangle  
\end{equation*}
with $A:L^{2}(0,T;V)\rightarrow{}L^{2}(0,T;V^\star)$:
\begin{equation*}
     \langle (Aq)(t),\phi(t)\rangle  = \mu\int_{\Omega} \nabla{\phi}(t) \cdot \nabla{q}(t)  \,d\Omega - \int_{\Gamma_{d}} q(t)\phi(t) \,d\Omega,
\end{equation*}
while the control operator $B:\mathcal{U}\times L^{2}(0,T;V)\rightarrow{}L^{2}(0,T;V^\star)$ is defined as
\begin{equation*}
     \langle B(u(t),q(t)),\phi(t)\rangle  = \int_{\Omega_{c}} u(t)\nabla{q}(t) \cdot \nabla{\phi}(t) \,d\Omega
\end{equation*}
We can now prove an existence theorem showing that there exist at least an optimal solution to the time-dependent problem. We have the following
\begin{theorem}[Existence of an optimal control]\ \\ 
        Suppose that the assumptions of  {\em Proposition} \ref{Prop_1} hold, 
        then there exists at least an optimal solution $(\hat{q},\hat{u})\in Y \times \mathcal{U}$ of the optimal control problem \eqref{functional_td_analysis}- \eqref{pde_td_analysis}.
\begin{proof} (Sketch)
    We use Theorem 9.4 of  \cite{manzonioptimal} (see Theorem 3 above, and its proof). 
    Hypotheses $(i)$, $(ii)$, $(iv)$ are trivially satisfied. Notice that \eqref{Y_norm} bounds the $\|\cdot\|_{H^1(0,T;V,V^\star)}$-norm.
    
    Hypothesis $(ii)$ asks to show that the set of feasible points is $weakly^\star$ sequentially closed in $Y\times \mathcal{U}$. The following sequences are considered:
    \begin{align*}
    \begin{split}
    &u_{n} \overset{\star}{\rightharpoonup} u \quad \textrm{$in$  $\mathcal{U}$}, \quad q_{n} \rightharpoonup q \quad \textrm{$in$ $Y$}, \quad Aq_{n} \overset{\star}{\rightharpoonup} \Xi \quad in\: L^2(0,T;V^\star), \quad B(u_{n},q_{n})  \overset{\star}{\rightharpoonup} \Phi \quad in\: L^2(0,T;V^\star)
    \end{split}
    \end{align*}
    with $G(q_{n},u_{n}) = 0\quad \forall n\in\mathbb{N}$.
    $(q,u)\in Y\times \mathcal{U}$ since $\mathcal{U}$ is $weakly^\star$ sequentially closed. The proof that $G(q,u)=0$ in $W$ is done in three substeps:
     \begin{itemize}
        \item $Aq=\Xi$ since $\nabla{q_{n}}\rightharpoonup\nabla{q}$ $in$ $L^2(0,T;L^2(\Omega))$ and $q_{n}\rightharpoonup q$ $in$ $L^2(0,T;L^2(\Gamma_{d}))$.
        \item $B(u,q)=\Phi$, since, $\forall \phi\in L^2(0,T;V)$
        \begin{equation*}
            \int_{0}^{T} \langle B(u_{n}(t),q_{n}(t)),\phi(t)\rangle  dt \rightarrow \int_{0}^{T} \langle B(u(t),q(t)),\phi(t)\rangle  dt,
        \end{equation*}
        that is,
        \begin{equation} \label{B_equal_phi}
            \int_{0}^{T} \int_{\Omega_{c}} (u_{n}(t)\nabla{q_{n}(t)}-u(t)\nabla{q}(t))\cdot \nabla{\phi}(t) \,d\Omega dt  \rightarrow 0,
        \end{equation}
        adding and subtracting to \eqref{B_equal_phi} $u\nabla{q}_{n}\cdot\nabla{\phi}$ it suffices to note that     
        \begin{equation*}
            \lim_{n\to \infty} \int_{0}^{T}\int_{\Omega_{c}} ( u(t) - u_{n}(t) ) \nabla{q}(t) \cdot \nabla{\phi}(t) \,d\Omega dt = 0 
        \end{equation*}
        
        that can be proven exploiting the compact embedding $H^{2}(\Omega_{c})\hookrightarrow\hookrightarrow L^{\infty}(\Omega_{c})$ in dimension $n=2,3$.
         In fact, using Aubin-Lions Lemma (see e.g. \cite{AubinLions,Chen_2013,Simon1986}) holds the following embedding
        $\mathcal{U}\hookrightarrow \hookrightarrow L^{\infty}(0,T;L^{\infty}(\Omega_{c}))$. We then have:
        \begin{align*}
        \begin{split}
            &\left| \int_{0}^{T}\int_{\Omega_{c}} (u(t)-u_{n}(t))\nabla{q}_{n}(t)\cdot{\nabla{\phi}}(t) \,d\Omega dt \right| \leq \\
            &\leq \|u-u_{n}\|_{L^{\infty}(0,T;L^{\infty}(\Omega_{c}))}\|\nabla{\phi}\|_{L^2(0,T;L^2(\Omega))}\|\nabla{q}_{n}\|_{L^2(0,T;L^2(\Omega))} \rightarrow 0
        \end{split}
        \end{align*}
        where $\|u-u_{n}\|_{L^{\infty}(0,T;L^{\infty}(\Omega_{c}))}\rightarrow 0$. Notice that , being a weak convergence sequence, $\|\nabla{q}_{n}\|_{L^2(0,T;L^2(\Omega))}$ is bounded (see \cite[Proposition A.3.]{manzonioptimal}). 
        Similarly, 
        \begin{equation*} 
            \int_{0}^{T}\int_{\Omega_{c}} ( \nabla{q_{n}}(t) - \nabla{q}(t) ) \cdot \nabla{\phi}(t)u(t) \,d\Omega dt 
        \end{equation*}
        vanishes using the fact that 
        $u\in \mathcal{U}\subseteq L^{\infty}(0,T;L^{\infty}(\Omega_{c}))$ and then $\nabla{\phi}u\in L^2(0,T;L^2(\Omega))$, and since $\nabla{q}_{n}\rightharpoonup \nabla{q}$ in $L^2(0,T;L^2(\Omega)$.
        \item 
        \begin{equation} \label{q_n_to_q_dot}
            \int_{0}^{T} \langle\dot{q}_{n}(t),\phi\rangle dt \rightarrow \int_{0}^{T} \langle\dot{q}(t),\phi\rangle dt \qquad \phi\in L^2(0,T;V)
        \end{equation}
        integrating by parts with: $\phi = w\psi$, $w\in V,\psi\in C^{1}[0,T]$, $\psi(T)=0$, we derive
        \begin{equation*}
            \int_{0}^{T} \langle\dot{q}_{n}(t),\phi\rangle dt = -\int_{0}^{T} (q_{n},w)_{L^2(\Omega)}\psi'(t) dt - (q_{n}(0),w)_{L^2(\Omega)}\psi(0),
        \end{equation*}
        analogously,
        \begin{equation*}
            \int_{0}^{T} \langle\dot{q}(t),\phi\rangle dt = -\int_{0}^{T} (q,w)_{L^2(\Omega)}\psi'(t) dt - (q(0),w)_{L^2(\Omega)}\psi(0);
        \end{equation*}
        \eqref{q_n_to_q_dot} is proved once we notice that since $q_{n} \rightharpoonup q$ in $Y$, we have that $q_{n}(0) \rightharpoonup q(0)$ in $L^2(\Omega)$, furthermore $w\psi'\in L^2(0,T;V)$ and $w\psi(0)\in V$. 
        \end{itemize}
        
    Gathering all this information together we can conclude that $G(q,u)=0$ in $W=V^{\star}$; by the arbitrariness of the sequence taken, the set of feasible points is $weakly^\star$ sequentially closed.
\end{proof}
\end{theorem}

Due to the nature of $\mathcal{U}$, $Y\times \mathcal{U}$ is not a reflexive space, however, being the dual of the separable space $H^1(0,T;V^\star,V)\times \Upsilon$, the theorem can still be applied, where the space $\Upsilon$ is defined as
\begin{equation*}
    \Upsilon= \{f\in L^{1}(0,T;H^2(\Omega_{c})^{\star}): \frac{\partial{f}}{\partial{t}}\in L^2(0,T;H^2(\Omega_{c})^{\star\star})\}.
\end{equation*}

%% file: sections/num_sim.tex
%In this Section, we present several numerical tests both in the transient and steady-state regimes and for different layouts of the cloak and the obstacle. After numerically checking that solving the optimality system (\ref{opt_discrete}) allows converging to the steady-state solution at the final instant, we investigate the performance of the reduction technique in capturing the dynamics of the  \emph{parametrized} OCP arising from the cloaking objective. 

In this Section, after discretizing both in space and in time the OCP, we present the main numerical results. We address both the steady state and the transient regime, considering different layouts of the obstacle and of the cloak.

\subsection{Numerical Discretization of the OCP}

For the numerical approximation of the reference problem and the OCPs, we employ the finite element method. We introduce a triangulation of the domain $\Omega_{un}$ and select piecewise linear, globally continuous ansatz functions $\phi_i$ ($\mathbb{P}_1$ finite elements) for the space approximation of state and adjoint variables in $\Omega$, while the basis is truncated to $\Omega_c$ for the control approximation. 

To generate a suitable triangulation on $\Omega$, we restrict the previous mesh to those elements belonging to $\Omega_{un} \setminus \Theta$, so that the nodes that belong to the obstacle are masked and a Dirichlet boundary is introduced instead. An illustrative example of the computational mesh and the resulting domain definition is shown in Figure \ref{control_layout}. We define by $E\in \R^{N_q \times N_z}$ the matrix that restricts the reference nodes to the OCP mesh nodes, where $N_q$ and $N_z$ are the dimension of the state and reference variables, respectively. In this way, a matrix $A\in \R^{N_z \times N_z}$ assembled on the reference mesh can be restricted to the OCP mesh through the projection $\tilde{A} = E A E^{\top}$.
In the following, we indicate with a tilde the matrices restricted to the computational domain of the control problem, e.g, $\tilde{A} = EAE^{\top}$.
The FEM approximation of the reference dynamics is therefore
\begin{equation} \label{reference_FEM}
    \begin{cases}
    M \dot{\bb{z}} + A\bb{z} = \bb{F}  \, , \, t \in (0,T) \\
    \bb{z}(0)                = \bb{0}  \, , \\ 
    \end{cases}
    \tag{\ref{reference_state} }
\end{equation}
where $M_{ij} = \int_{\Omega} \phi_i \phi_j d\Omega$, $A_{ij} = \int_{\Omega} \mu \nabla \phi_i \cdot \nabla \phi_j d\Omega$ and $F_i = \int_{\Omega} s \phi_i d\Omega$ for $i,j = 1\, \ldots, N_z$. Regarding the OCP discretization, we follow an Optimize-then-Discretize (OtD) strategy (see e.g. \cite{herzog}) that consists of discretizing the optimality conditions obtained at the continuous level.
We discretize the state equation:
\begin{equation*} \label{state equation}
    \begin{cases}
    EME^{T}\dot{\textbf{q}}(t)+ \{E[A + A_{r}]E^{T} + \bar{\bar{B}}_{u}\textbf{u}(t)+\bar{\bar{B}}_{f}\textbf{f}(t)+\bar{\bar{B}}_{v}\textbf{v}(t) \}\textbf{q}(t) = E\textbf{F} + \textbf{F}_{o} \quad \text{$t\in(0,T)$}\\
    \textbf{q}(0) = \textbf{0} 
    \end{cases}\,
\end{equation*}
introducing the Robin matrix  $(A_{r})_{i,j}=\int_{\Gamma_{d}} -\varphi_{j}\varphi_{i} \,d\Gamma$, the vector $\bb{F}_{o}$ including the contribution arising from the Dirichlet data, and a rank 3 tensor for each control variable:
\begin{equation*}
\begin{cases}
    (\bar{\bar{B}}_{u})_{i,j,k} = \int_{\Omega_{c}} (\varphi_{k}U)\nabla{\varphi_{i}} \cdot \nabla{\varphi_{j}} \,d\Omega \\
    (\bar{\bar{B}}_{f})_{i,j,k} = \int_{\Omega_{c}} (\varphi_{k}L)\nabla{\varphi_{i}} \cdot \nabla{\varphi_{j}} \,d\Omega \\ 
    (\bar{\bar{B}}_{v})_{i,j,k} = \int_{\Omega_{c}} (\varphi_{k}S)\nabla{\varphi_{i}} \cdot \nabla{\varphi_{j}} \,d\Omega; 
\end{cases}
\end{equation*}
\noindent each tensor vector product gives us a matrix of the type:
\begin{equation*}
    (\bar{\bar{B}}_{u}\textbf{u}(t))_{i,j} = \sum_{k=1}^{N_{u}} u_{k}(t)\int_{\Omega_{c}} \varphi_{k} \nabla{\varphi_{i}} \cdot \nabla{\varphi_{j}} \,d\Omega = \sum_{k=1}^{N_{u}} (\bar{\bar{B}}_{u})_{i,j,k}u_{k}(t).
\end{equation*}

The time discretization is obtained by means of the Backward Euler method, however, we underline that the approximation scheme can be easily converted into a more generic $\theta-method$.

The discretized OCP yields the following Nonlinear Programming Problem (NLP): 
\begin{equation*}
    min \quad \tilde{J}( \textbf{q}_{0},...,\textbf{q}_{N},\textbf{u}_{0},...,\textbf{u}_{N},\tilde{\textbf{f}}_{0},...,\tilde{\textbf{f}}_{N},\textbf{v}_{0},...,\textbf{v}_{N})
\end{equation*}
\begin{equation} \label{NLP_formulation_state_equation}
\hspace*{-0.2cm}
   \textrm{subject to} 
    \begin{cases}
    \tilde{A}_{+}(\textbf{u,f,v}_{i+1})\textbf{q}_{i+1} =  \tilde{A}_{-}(\textbf{u,f,v}_{i})\textbf{q}_{i} + E\textbf{F}  + \textbf{F}_{o,i+1}, \  i=0,\ldots,N-1\\
    \textbf{q}_{0} = \textbf{0} \\
    2\mu + (\textbf{u}_{i})_{j} + (\textbf{f}_{i})_{j} \geq \epsilon > 0,\ i=0,...,N, j=1,...,N_{u} \\ 
    (\mu + (\textbf{u}_{i})_{j})(\mu + (\textbf{f}_{i})_{j}) - (\textbf{v}_{i})_{j} ^ 2 \geq \epsilon > 0, \  i=0,\ldots,N, j=1,\ldots,N_{u}.
    \end{cases}
\end{equation}
In particular: the last two equations represent the constraints that have to be satisfied at each time instant to keep the diffusivity matrix \eqref{diffusivity_matrix} positive definite; $\epsilon>0$ is a tolerance imposed on the constraints; and the transition matrices $\tilde{A}_{+}$, $\tilde{A}_{-}$ are derived similarly to what is done in a similar context in  \cite{9787068} by Sinigaglia et al.
The adjoint equation and the expression of the gradient of the cost functional are derived using the discrete adjoint method \cite{discretelagrangian} so that, by introducing the adjoint variables  $\textbf{p}_{i}$, $i=0,\ldots,N$, we obtain the following \textit{discrete Lagrangian}:
\begin{equation} \label{discrete_lagrangian}
    \tilde{\mathcal{L}} = \tilde{J} - \Delta{t}\{ \sum_{i=0}^{N-1} \langle\tilde{A}_{+}\textbf{q}_{i+1} - (\tilde{A}_{-}\textbf{q}_{i} +E\textbf{F} + \textbf{F}_{o,i+1}), \textbf{p}_{i+1} \rangle  \}
\end{equation}
The gradient of the cost functional with respect to the control variables is obtained differentiating \eqref{discrete_lagrangian} with respect to each of the control variables $\textbf{u}_{i},\:\textbf{f}_{i},\:\textbf{v}_{i}$. In particular, the derivatives of \eqref{discrete_lagrangian} with respect to $\textbf{q}_{i}$, $i=0,\ldots,N$, give us the (backward in time) adjoint equation; see \cite[Chapter 8]{manzonioptimal} for further details on the way the cost functional can be approximated.

 Since  we also provide numerical results concerning the steady-state regime, we report its numerical discretization below. 
The reference field, before given by \eqref{reference_FEM} is obtained by solving:
\begin{equation}
\label{reference_FEM_steady_stae}
     A\bb{z} = \bb{F}. 
\end{equation}
While the approximation of the optimality system made by the state equation \eqref{ocp_formulation_pde}, the adjoint equation \eqref{adjoint_problem} and the Euler equations \eqref{euler_equation} becomes:

\begin{equation} \label{equ:15122}
    \begin{cases}
     \{E[A + A_{r}]E^{T} + \bar{\bar{B}}_{u}\textbf{u}+\bar{\bar{B}}_{f}\textbf{f}+\bar{\bar{B}}_{v}\textbf{v} \}\textbf{q} = E\textbf{F} + \textbf{F}_{o} + \textbf{F}_{ocp}\\ 
       \{E[A + A_{r}]E^{T} + \bar{\bar{B}}_{u}\textbf{u}+\bar{\bar{B}}_{f}\textbf{f}+\bar{\bar{B}}_{v}\textbf{v} \}^T\textbf{p} = M_{obs}(\textbf{q}-E\textbf{z})\\ 
        \beta M_{u}\textbf{u} + \beta_{g} A_{u}\textbf{u} - (\bar{\bar{B}}_{u}\textbf{p})\textbf{q} = \textbf{0} \\ 
        \xi M_{u}\textbf{f} + \xi_{g} A_{u}\textbf{f} - (\bar{\bar{B}}_{f}\textbf{p})\textbf{q} = \textbf{0} \\ 
        \gamma M_{u}\textbf{v} + \gamma_{g} A_{u}\textbf{v} - (\bar{\bar{B}}_{v}\textbf{p})\textbf{q} = \textbf{0}. 
    \end{cases}
\end{equation}
In the steady-state case, the constraints in \eqref{NLP_formulation_state_equation} become:
\begin{equation*}
\begin{cases}
    2\mu + \textbf{u}_{j} + \textbf{f}_{j} \geq \epsilon > 0 \quad j=1,\ldots,N_{u}\\
    (\mu + \textbf{u}_{j})(\mu + \textbf{f}_{j}) - (\textbf{v}_{j})^2 \geq \epsilon > 0 \quad j=1,\ldots,N_{u}. \\
    
\end{cases}
\end{equation*}

\subsection{Simulation results, steady-state case}
\label{Steady_state_numerical_simulation}
We start considering the cloaking problem under stationary condition \eqref{SS_ocp_functional}--\eqref{SS_pde_formulation}. 
The steady-state OCP is solved using the ${MATLAB}$\textsuperscript{\textregistered} function \texttt{fmincon}, once all the FEM matrices have been built through the \texttt{RedbKit} library \cite{redbKIT}. The maximum length of the sides of the triangles in the triangulation is $h_{max}=0.1232$ for a number of elements equal to 2320. 
The \texttt{Tensor} Toolbox (Bader and Kolda (2006)) \cite{10.1145/1186785.1186794,doi:10.1137/060676489} is used to perform efficient tensor computations. In the functional \eqref{SS_ocp_functional} we set as parameters:
$\beta=10^{-9}$, $\beta_{g}=7\times 10^{-6}$, $\xi=10^{-9}$, $\xi_{g}=7\times 10^{-6}$, $\gamma=10^{-9}$ and $\gamma_{g}=5\times10^{-5}$, in the PDE \eqref{SS_pde_formulation} we pose: $T_{0}=0^\circ C$, $\mu = 1$,  $\mu_{min}=0.1$, 
 $\alpha=1$, $s =100$ and the tolerance on the two constraints is of $\epsilon = 10^{-3}$. 
Figure \ref{ss_pb1} shows that the cloaked and reference temperature field are indistinguishable at the right of the obstacle, while the left side is more difficult to cloak and is possible to remark some difference between the two. The three controls are reported in Figure \ref{ss_pb1}. By these plots, we can infer that the control variables play a different role in the optimization, and that each of them is relevant to achieve proper cloaking.

\begin{figure}[t!]
\centerline{
\includegraphics[width=0.5\textwidth]{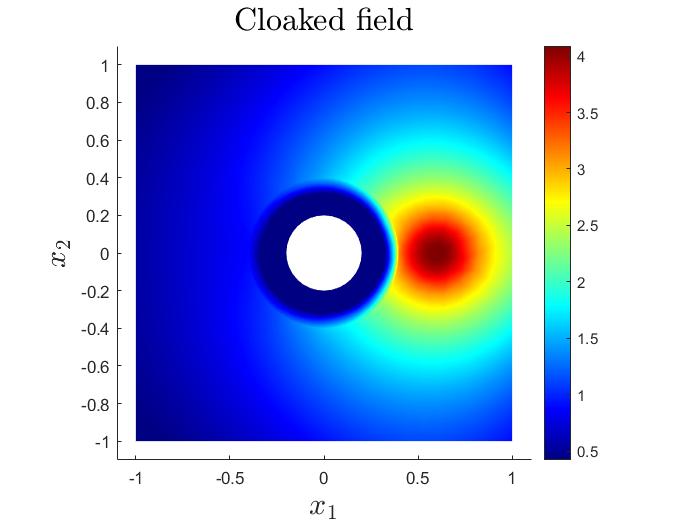}
\includegraphics[width=0.5\textwidth]{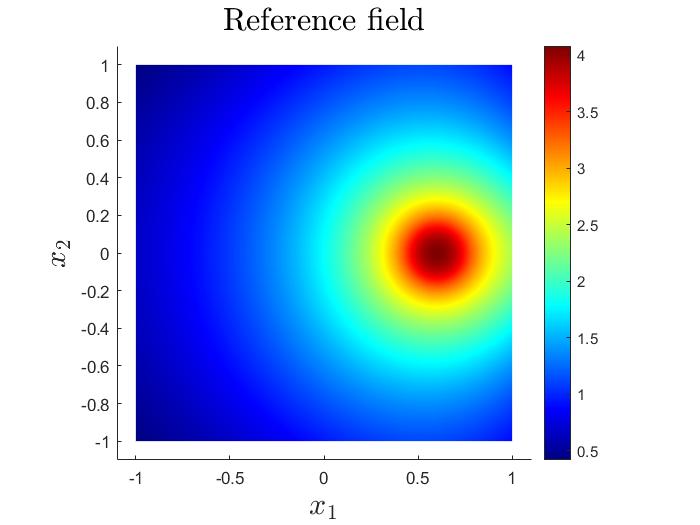}}
%\hspace{-0.25cm}
\centerline{
\includegraphics[width=0.365\textwidth]{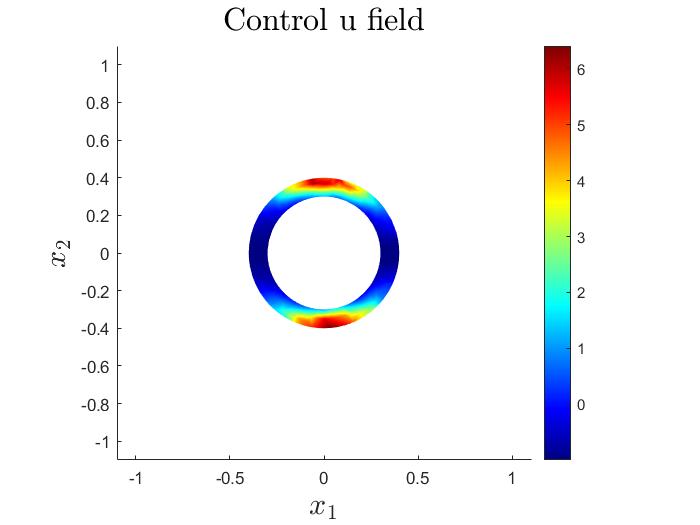}
\includegraphics[width=0.365\textwidth]{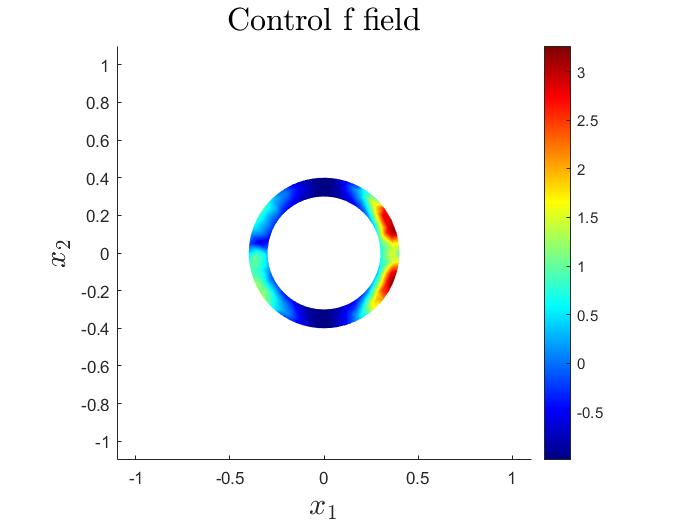} \includegraphics[width=0.365\textwidth]{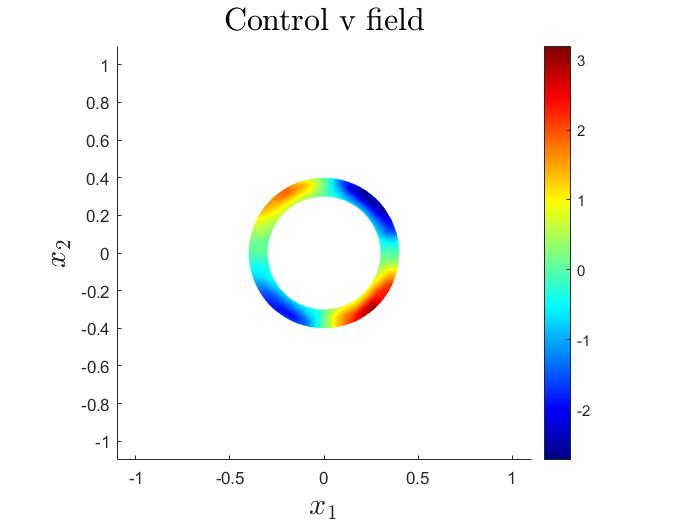}}

\caption{Steady-state, passive thermal cloaking: cloaked and reference fields (first row), optimal controls $u$, $f$, and $v$ (second row).}

\label{ss_pb1}
%\hspace{-0.45cm}
\end{figure}

\begin{figure}[t!]

\centerline{
\includegraphics[width=0.5\textwidth]{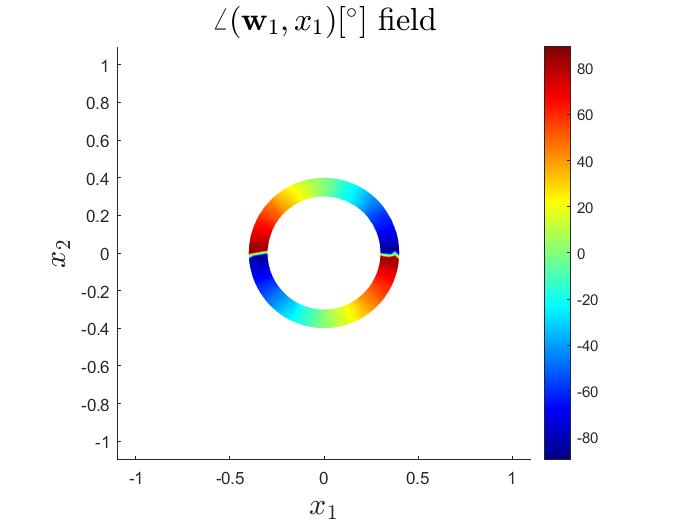}
\includegraphics[width=0.5\textwidth]{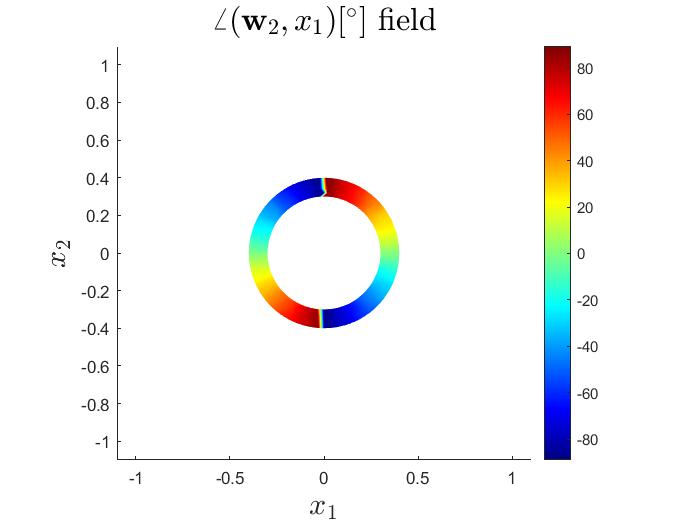}}

\caption{Steady-state, passive thermal cloaking: angle (in degrees) between each of the two eigenvectors and the horizontal axis.}
\label{angle_eigenvetors}

\end{figure}

Figure \ref{angle_eigenvetors} reports the inclination angle of the eigenvectors $\textbf{w}_{1}$ and $\textbf{w}_{2}$ of the diffusivity matrix $K$ \eqref{diffusivity_matrix}, representing the direction of maximum and minimum diffusivity, respectively. Since $K$ is symmetric, eigenvectors are orthogonal. 
Analogously, Figure \ref{eigenvalues} shows the eigenvalues $\lambda_{1}$ and $\lambda_{2}$: $\lambda_{1}$ is maximum where the flux lines are mostly affected by the cloak. We can remark that the control terms are able to steer the first eigenvector, that is the one associated with the maximum eigenvalue, changing the diffusivity inside the cloak in such a way that the heat flux is deviated to circumvent the obstacle. The need of employing three control terms in the diffusivity matrix stems from the fact that this structure allows the proper rotation of the eigenvectors of \eqref{diffusivity_matrix}.

\begin{figure}[t!]
\centerline{
\includegraphics[width=0.5\textwidth]{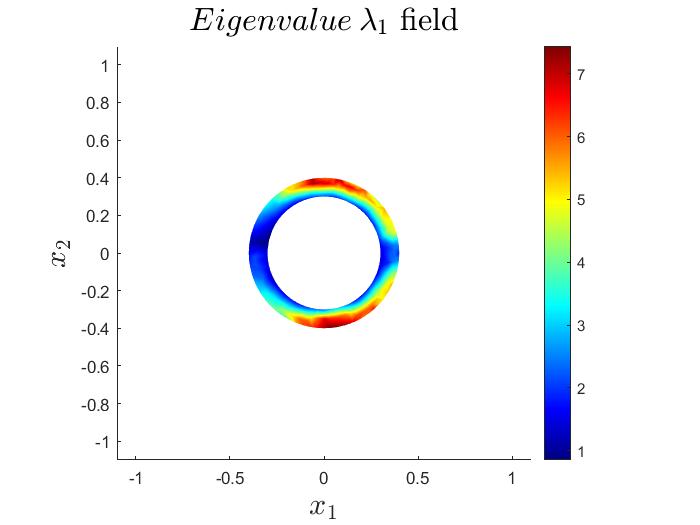}
\includegraphics[width=0.5\textwidth]{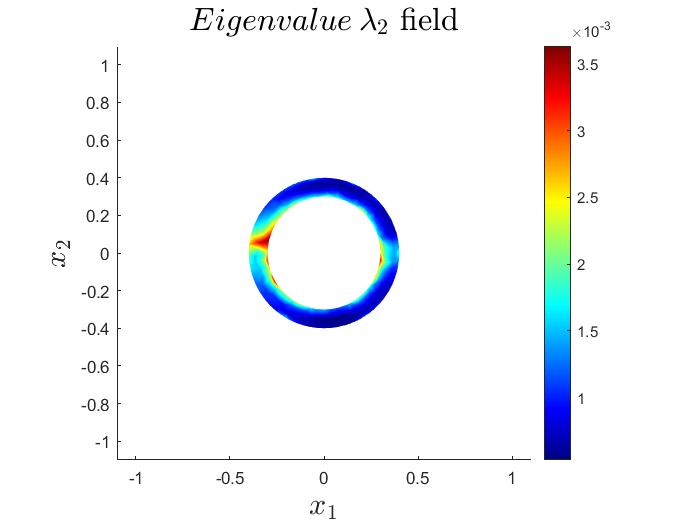}}

\caption{Steady-state, passive thermal cloaking: magnitude of the eigenvalues $\lambda_{1}$ (left) and $\lambda_{2}$ (right).} 

\label{eigenvalues}
\end{figure}

We now move towards the modification of the domain inside which the control is located. 
We tackle the cloaking of the silhouette of the half-woolen boar. The control terms are put in a thin offset of the target. This configuration still allows good cloaking performances since, despite the complexity of the object to hide, the controlled domain is connected, and this permits deviating the eigenvectors as desired. The results achieved can be seen in Figure \ref{passive_boar_cloak}. 

\begin{figure}[h!]
\centerline{
\includegraphics[width=0.365\textwidth]{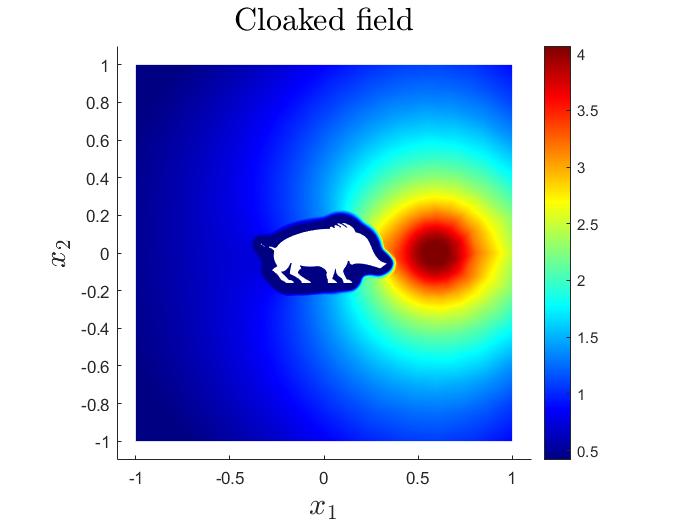}
\includegraphics[width=0.365\textwidth]{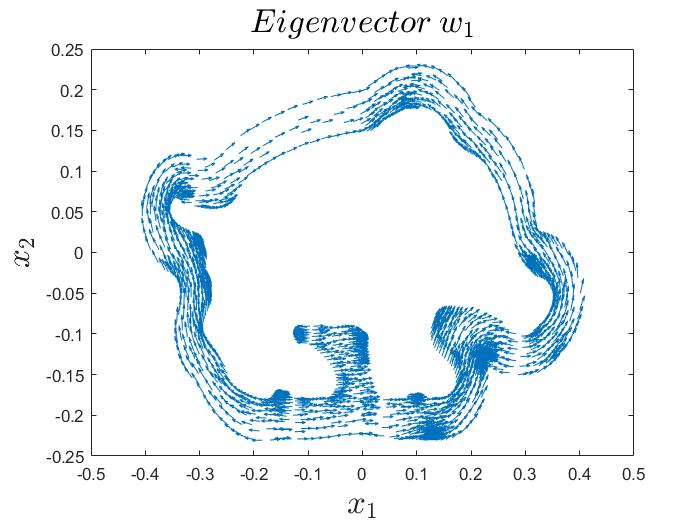}
\includegraphics[width=0.365\textwidth]{images/steady_state_fmincon/reference_stationary_field.jpg}}
\centerline{
\includegraphics[width=0.365\textwidth]{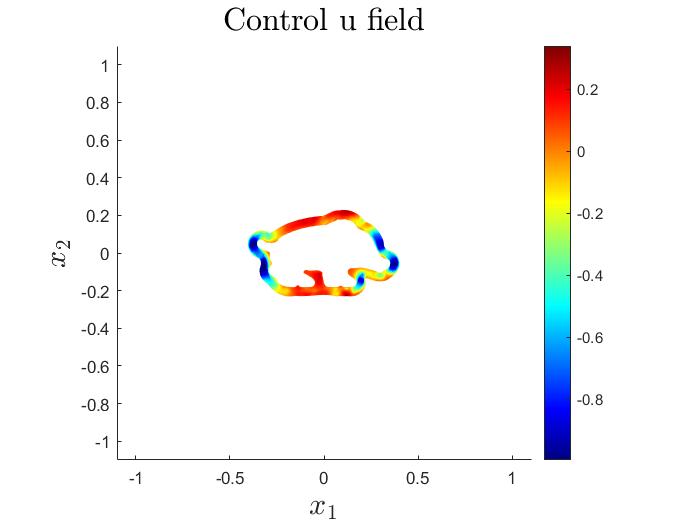}
\includegraphics[width=0.365\textwidth]{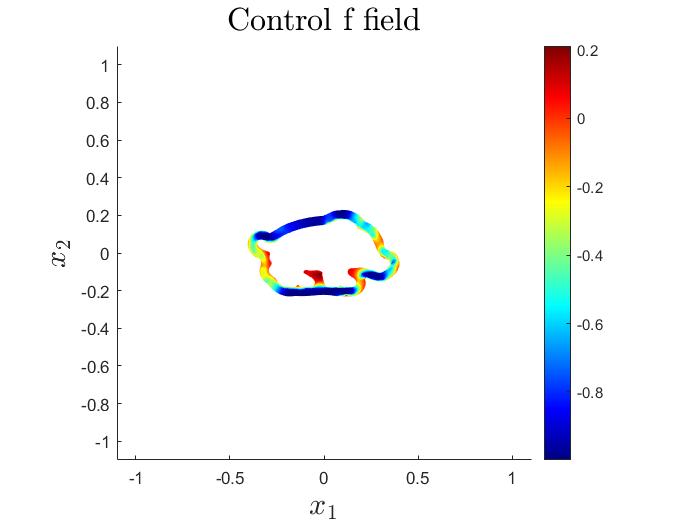}
\includegraphics[width=0.365\textwidth]{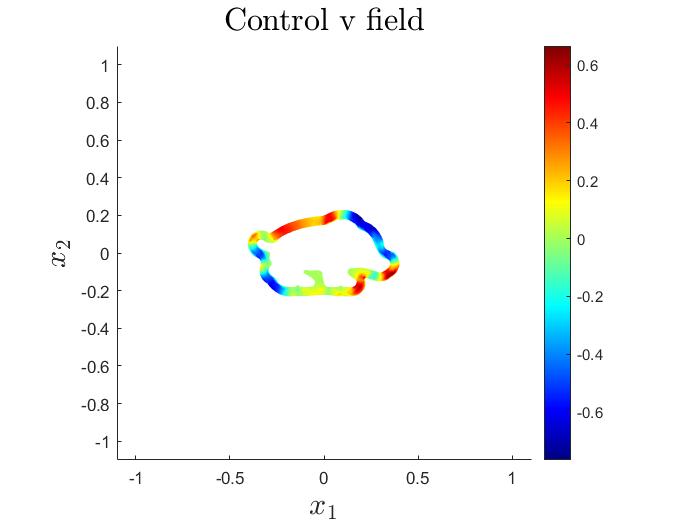}}

\caption{Steady-state, passive thermal cloaking: the case of a target object with arbitrarily complex shape. Here control variables are in a thin offset of the target.}
\label{passive_boar_cloak}

\end{figure}

Introducing the Mean Tracking Error (MTE) as:
\begin{equation*}
    \text{MTE} = \frac{\int_{\Omega_{obs}} |q^{ss}(x)-z^{ss}(x)|^2 d\Omega }{\int_{\Omega_{obs}} d\Omega },
\end{equation*}
we define the cloaking efficiency $\eta$ as in \textit{\cite{Sinigaglia_2022}}: 
\begin{equation*}
    \eta = \frac{|\text{MTE} - \text{MTE}^{\star}|}{\text{MTE}},
\end{equation*}
where MTE$^{\star}$ indicates the optimal variables while the absence of superscript denotes the uncontrolled case.
The resulting value of $\eta$ for each cloaking layout is shown in Figure \ref{tracking_errors_ss_passive}. Note that the region where the passive cloak struggles mostly is the one at the left of the obstacle.

\begin{figure}[h!]
\centerline{
\subfigure[$\eta=0.9$]{\includegraphics[width=0.5\textwidth]{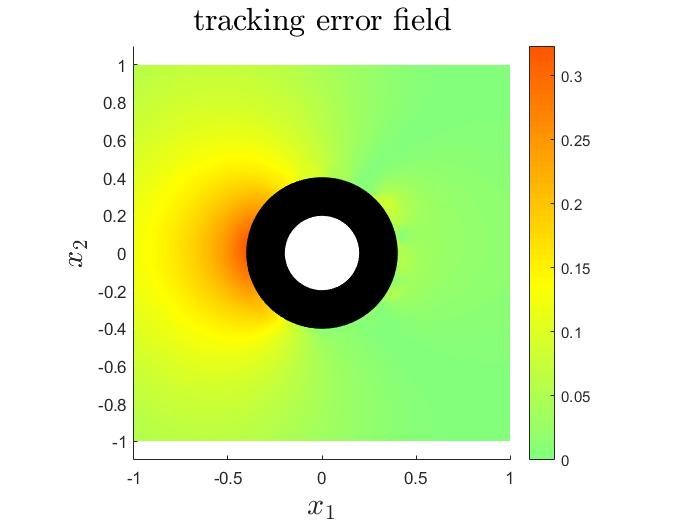}}
\subfigure[$\eta=0.8315$]{\includegraphics[width=0.5\textwidth]{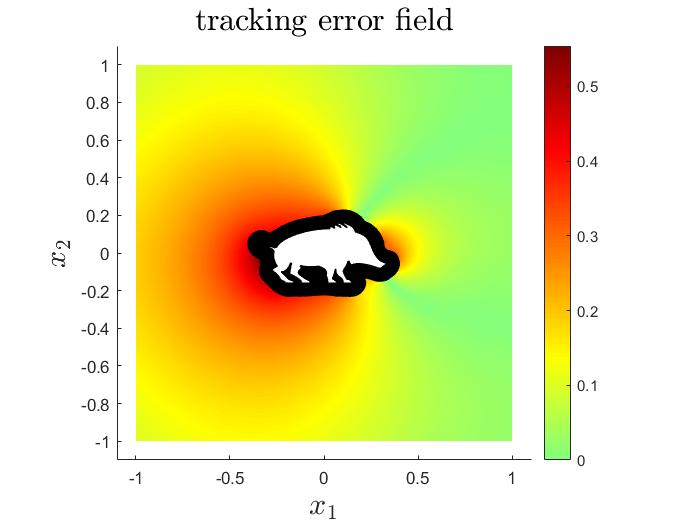}}
}
\caption{Steady-state, passive thermal cloaking: tracking error fields for the three considered layouts.}

\label{tracking_errors_ss_passive}
\end{figure}

Now we would like to characterize the behaviour of the cloak when used with a discretization that is finer than the one used for the optimization. Therefore, we simulate the cloaking problem on a finer grid than the one used for the above approximation.

We build a fine mesh dividing each triangle of the past triangulation into four triangles. More precisely, each side of each triangle of the coarse mesh has been halved.
The number of elements of the new mesh corresponds to 9280 and the maximum side length is $h_{max}=0.0616$. The number of control nodes goes from 255 (coarse grid) to 898 (fine grid), thus making the optimization process much more expensive. 
The coarse mesh, used for the optimization, and the refined domain in which the controls lie are shown in Figure \ref{coarse_fine_mesh}.
\begin{figure}[t!]
\centerline{
\includegraphics[width=0.5\textwidth]{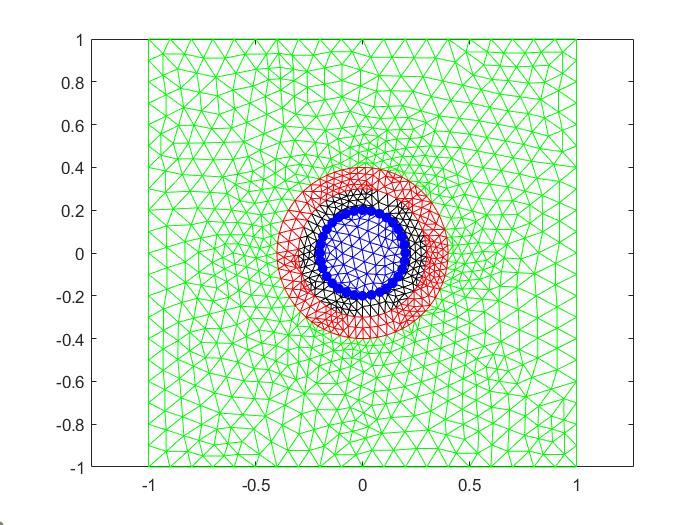}
\includegraphics[width=0.5\textwidth]{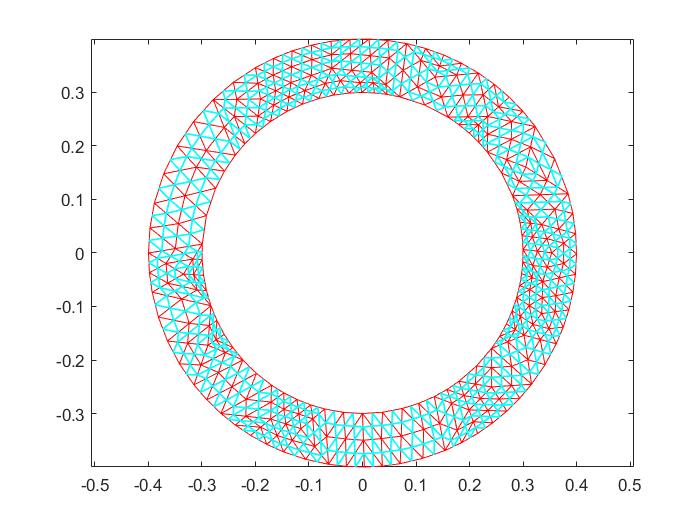}}

\caption{Left: coarse mesh that used to solve the OCP. Right: fine mesh over the domain $\Omega_{c}$. The vertices of the triangles shared with the coarse mesh are shown in red, while the newly added vertices are reported in cyan.} 
\label{coarse_fine_mesh}

\end{figure} The resolution of the optimal control problem from scratch has to be avoided, to do this, the optimal solution derived on the coarse mesh is interpolated in the fine mesh. The values assumed by the optimal control terms on the nodes of the coarse mesh are linearly extended to the nodes of the fine mesh that are not shared by the two meshes. This prolongation guarantees that the extended controls satisfy the nonlinear constraints. For the sake of comparison, the optimization problem in the finer mesh is solved from scratch. In Figure \ref{coarse_vs_fine_tracking_field} we report the tracking error fields in the two configurations, in the left picture the controls obtained from the optimization on the coarse mesh are applied, and in the right picture the controls have been recomputed on the fine grid.
\begin{figure}[t!]
\centerline{
\includegraphics[width=0.5\textwidth]{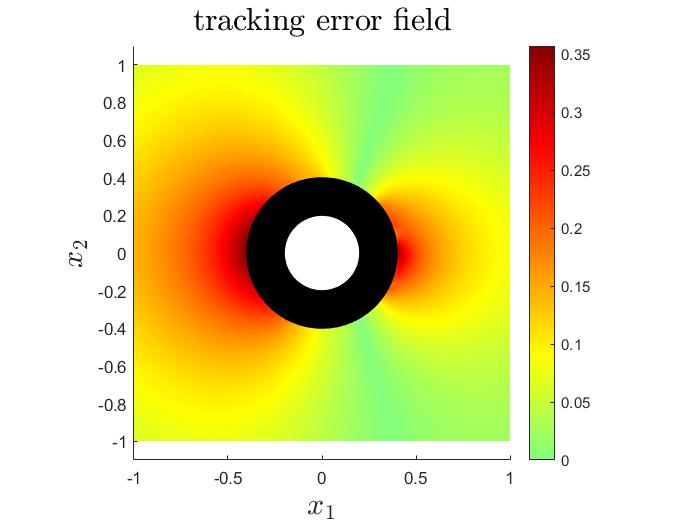}
\includegraphics[width=0.5\textwidth]{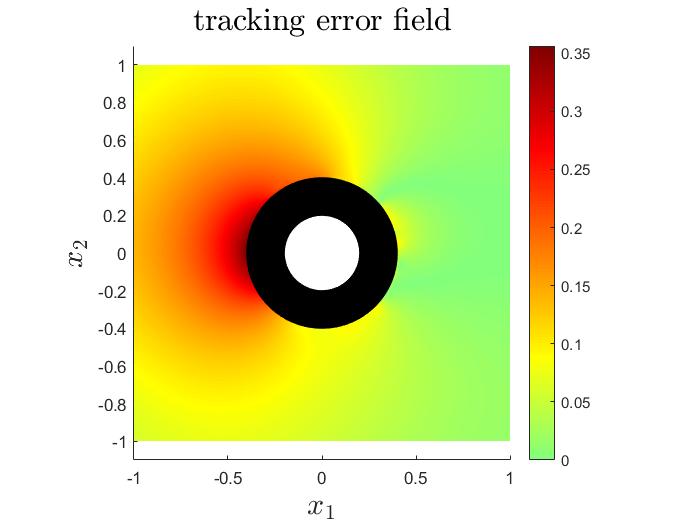}}
   
\caption{Steady-state, tracking error fields obtained with the optimal controls derived from the coarse mesh (left) and the optimal controls recomputed for the finer discretization (right).} 

\label{coarse_vs_fine_tracking_field}
\end{figure}
We can see that the two error fields are similar on the left portion of the obstacle.
On the opposite, the effects of the finer mesh are evident in the portion at the right of the obstacle. The reason for this difference is apparent comparing the controls utilized on this fine mesh. 
Therefore, in Figure \ref{controls_coarse_vs_fine} we report in the fine mesh the optimal controls derived on the coarse mesh (line 1) and the optimal controls derived resolving the optimization (line 2). The control terms present some differences, for example, the control term $f$ is more active in the optimal configuration and this explains the smaller tracking error at the right side of the obstacle. 
\begin{figure}[h!]
\centerline{
\includegraphics[width=0.365\textwidth]{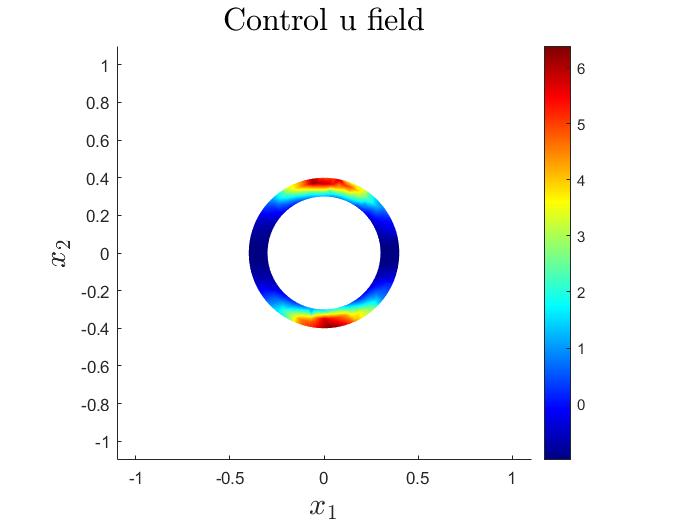}
\includegraphics[width=0.365\textwidth]{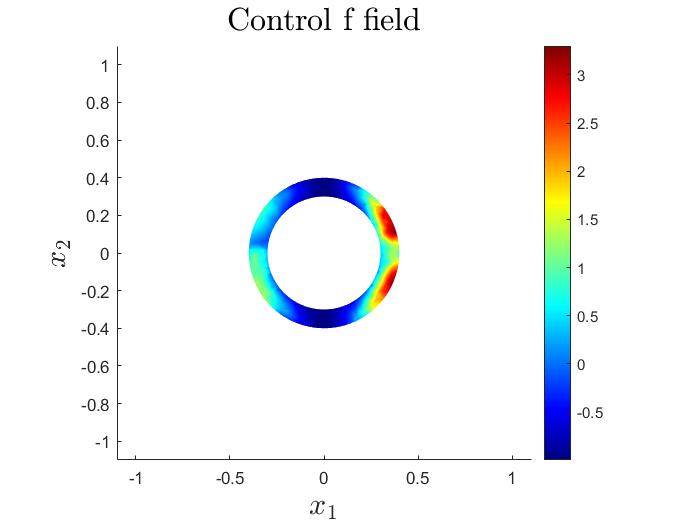}
\includegraphics[width=0.365\textwidth]{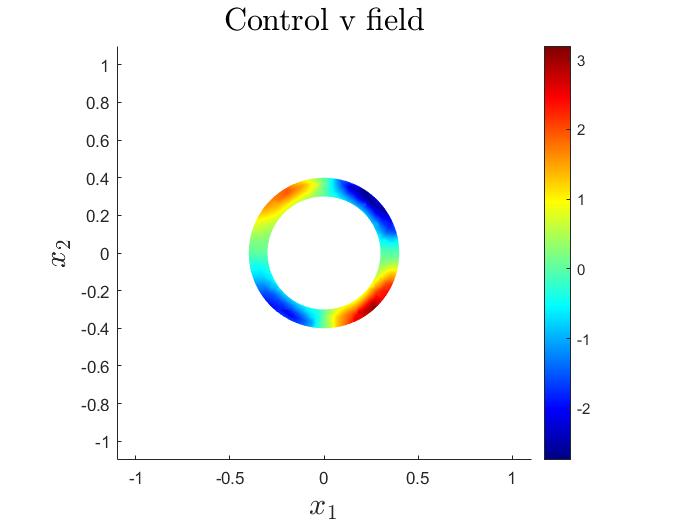}}
\centerline{
\includegraphics[width=0.365\textwidth]{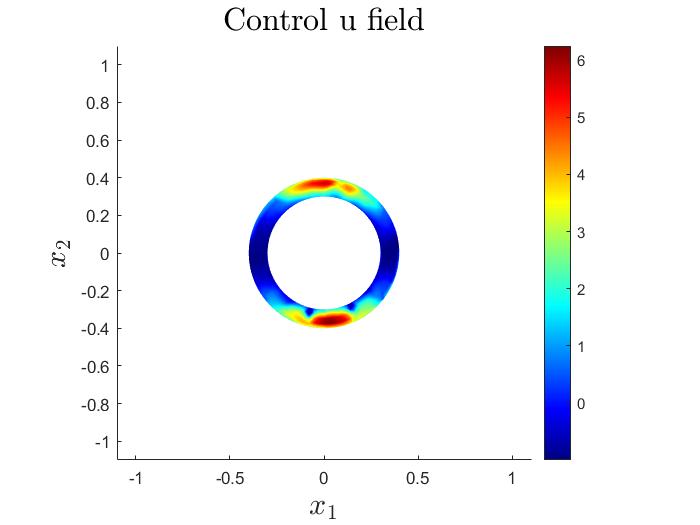}
\includegraphics[width=0.365\textwidth]{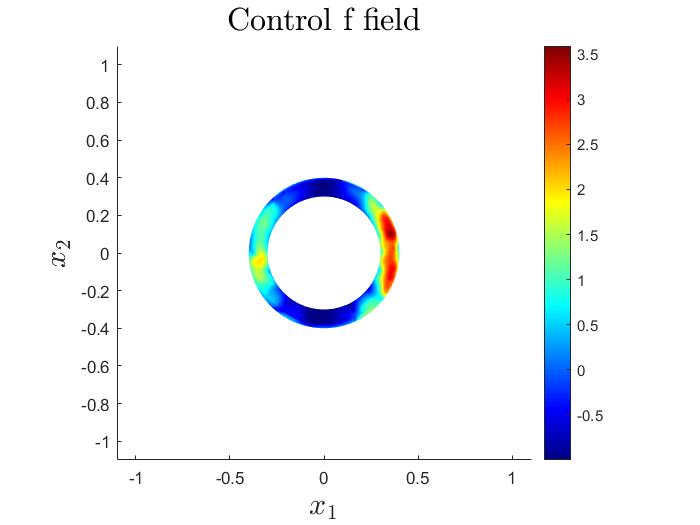}
\includegraphics[width=0.365\textwidth]{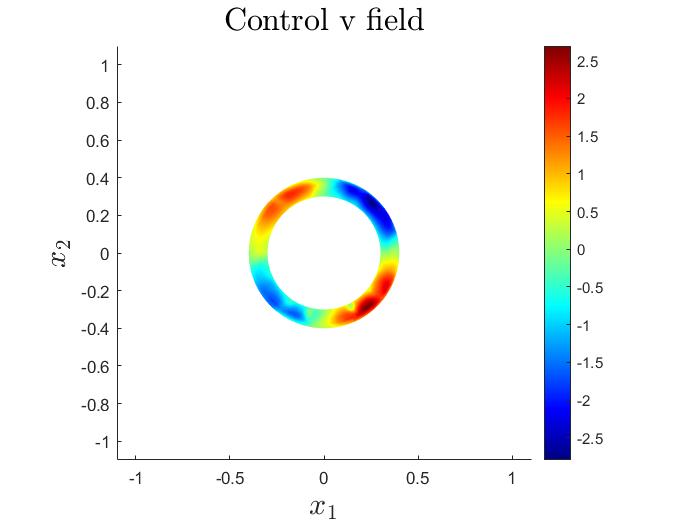}}
\caption{Steady-state, passive thermal cloaking, the controls derived on the coarse mesh and extended on the fine mesh are plotted in the first row, and the optimal controls derived on the fine mesh are shown in the second row. }

\label{controls_coarse_vs_fine}
\end{figure}

The tracking efficiency $\eta$ is 0.8568 applying coarse controls, while is 0.8729 applying the optimal controls. 
Despite this gap, the approach of extending a coarse control onto a finer discretization is still applicable and provides acceptable results with several computational benefits. 

We would like to investigate the behaviour of the cloak in front of changes in the probing field. That is, will the cloak designed for a first probing source be efficient in off-design conditions?
Numerical simulations suggest that, even in such situations, the performance of cloak are satisfactory. 
For the sake of clarity, the setup presenting a probing source at the right of the obstacle and with a circular cloak will be denoted as reference setup for the probing source. 
We consider a forcing term concentrated at the bottom of the object to be hidden, while maintaining the same distance between the center of the object and the center of the source. The source intensity $s$ is not varied. Figure \ref{tracking_field_bottom} shows the tracking error obtained using the optimal cloak corresponding to the control terms of the reference case and the tracking error derived with the optimal control terms recomputed for this new layout. 
\begin{figure}[t!]

\centerline{
\includegraphics[width=0.5\textwidth]{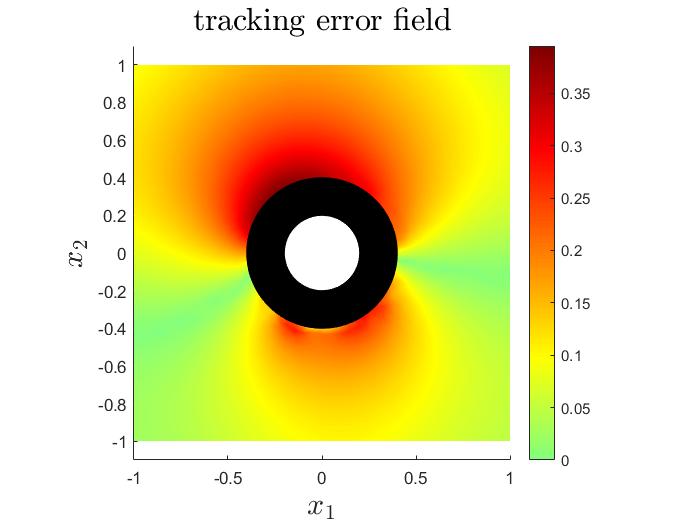}
\includegraphics[width=0.5\textwidth]{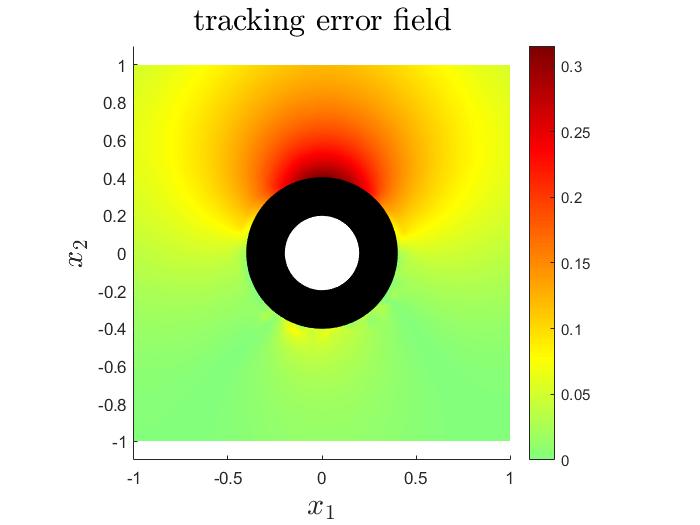}}
\caption{Steady-state, tracking error fields obtained with the optimal controls derived from the reference configuration (left) and the optimal controls recomputed for the problem at hand (right).} 

\label{tracking_field_bottom}
\end{figure} As is possible to see, the cloak is still efficient even in this new setting. In particular, the cloaking efficiency obtained with the controls of the reference configuration is $\eta=0.8143$, while the one obtained with the optimal controls for the problem at hand is $\eta=0.89$. In Figure \ref{controls_bottom}, we show the optimal controls derived from the new optimization problem. Comparing them with Figure \ref{ss_pb1}, is possible to observe that the control term $v$ is barely affected by the position of the source, on the other hand, the terms $u$ and $f$ are switching their role. Indeed, this reflects the part that $u$ and $f$ play in the diffusivity matrix $K$. 
\begin{figure}[t!]
\centerline{
\includegraphics[width=0.365\textwidth]{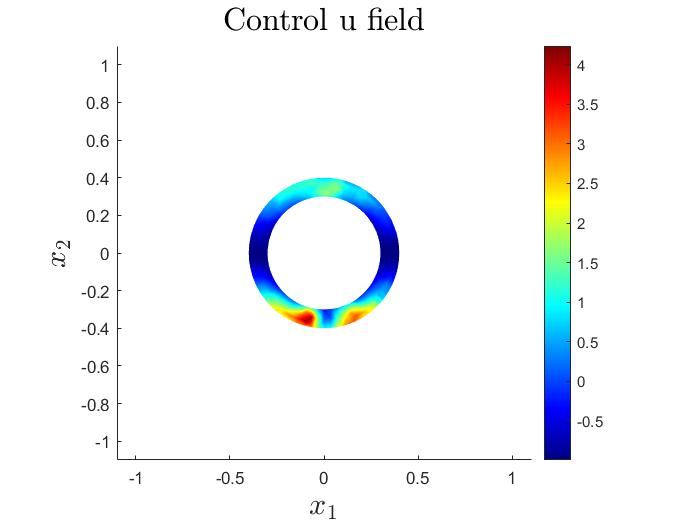}
\includegraphics[width=0.365\textwidth]{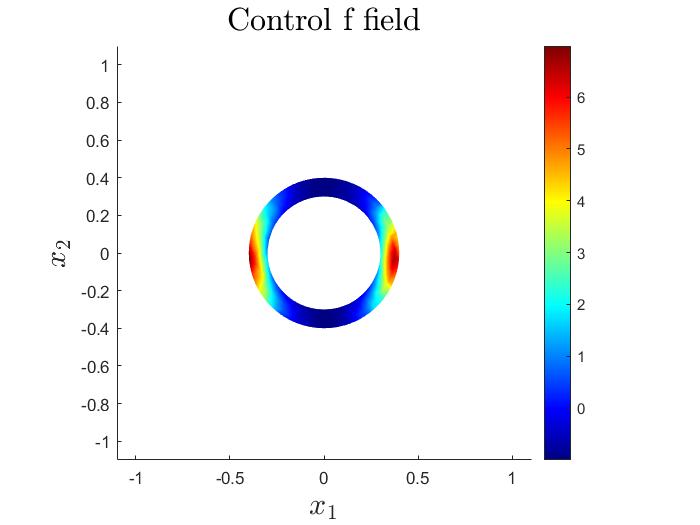}
\includegraphics[width=0.365\textwidth]{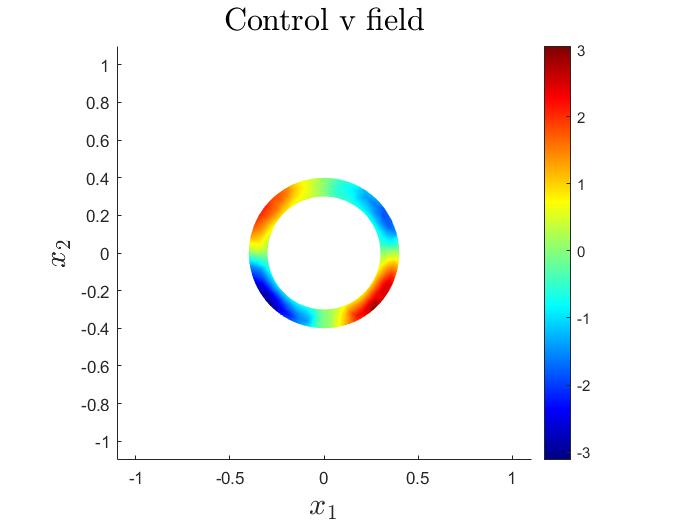}}

\caption{Steady-state, passive thermal cloaking, optimal controls $u$, $f$, $v$, with forcing term at the bottom of the obstacle.}

\label{controls_bottom}
\end{figure} The optimal angles of the diffusivity matrix are instead assuming the same pattern as in Figure \ref{angle_eigenvetors} and this is the reason why the optimal controls of the reference problem are still providing a sufficient cloaking effect. Ongoing numerical experiments seem to suggest that the reference setup remains reasonably efficient for other positions and magnitudes of the probing source, however this matter needs to be further investigated and will be part of future research.

\subsection{Simulation results, time-dependent case}

We can now move towards the approximation of passive thermal cloaking problems in the time-dependent case. 
 Considering a time-horizon $t_{f}=2$ and a time-discretization step $\Delta{t}=0.1429s$ ($N=14$), we end up with 3825 control variables. The main regularization parameters inserted into the cost functional are $\beta_{g}=\xi_{g}=\gamma_{g}=10^{-5}$. In the case of a distributed cloaking around a circular target object -- similarly to the first problem considered in the steady-state case -- a sequence of solutions at time instants $t=0.57(s)$, $t=1.29(s)$, $t=2.00(s)$ is shown in Figure \ref{state_time_dependent_solution}.  The control variables obtained at time instants $t=0.57(s)$, $t=1.29(s)$, $t=2.00(s)$ together with the first eigenvalue generated by them, are shown in Figure \ref{controls_time_dependent_case}, \ref{lambda_1_time_dependent_case}.

\begin{figure}[h!]
\centerline{
\includegraphics[width=0.365\textwidth]{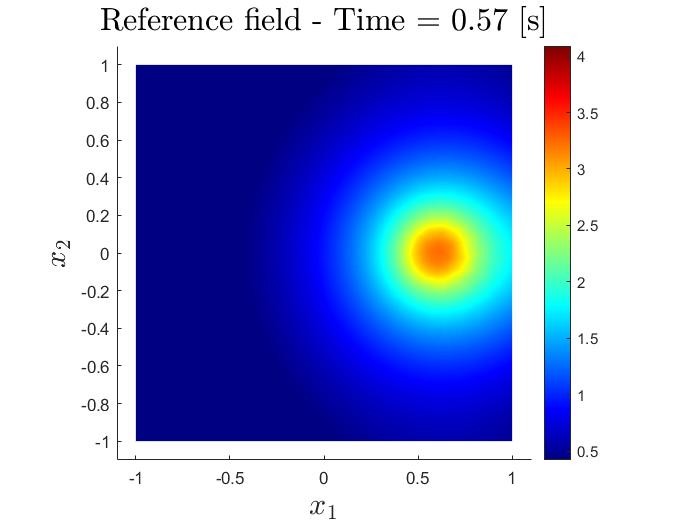}
\includegraphics[width=0.365\textwidth]{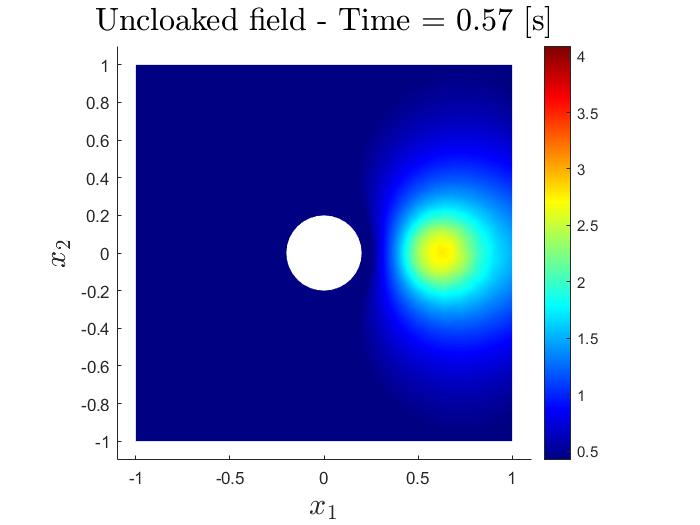}
\includegraphics[width=0.365\textwidth]{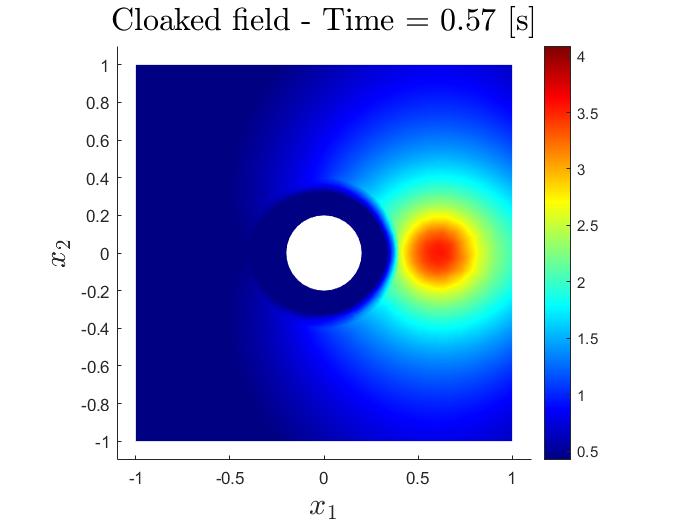}}
\centerline{
\includegraphics[width=0.365\textwidth]{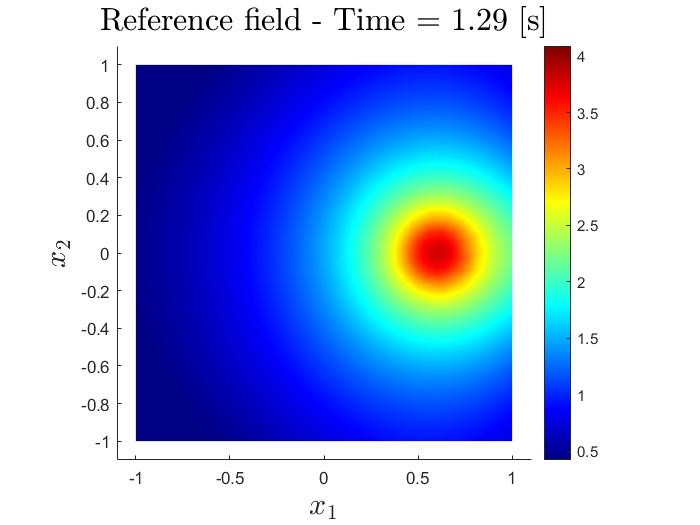}
\includegraphics[width=0.365\textwidth]{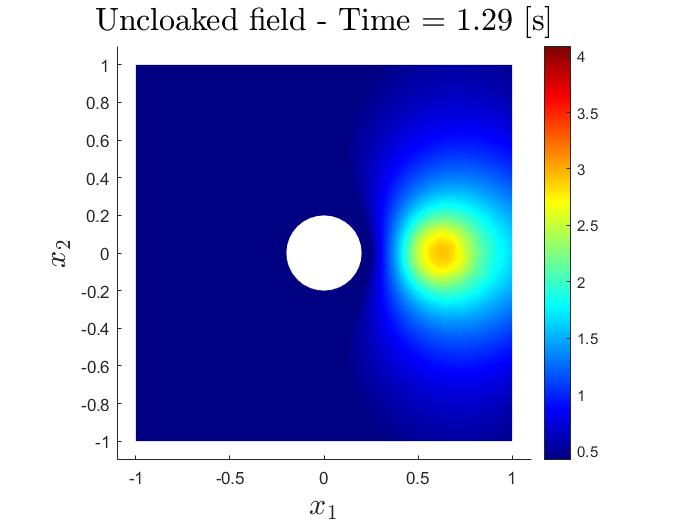}
\includegraphics[width=0.365\textwidth]{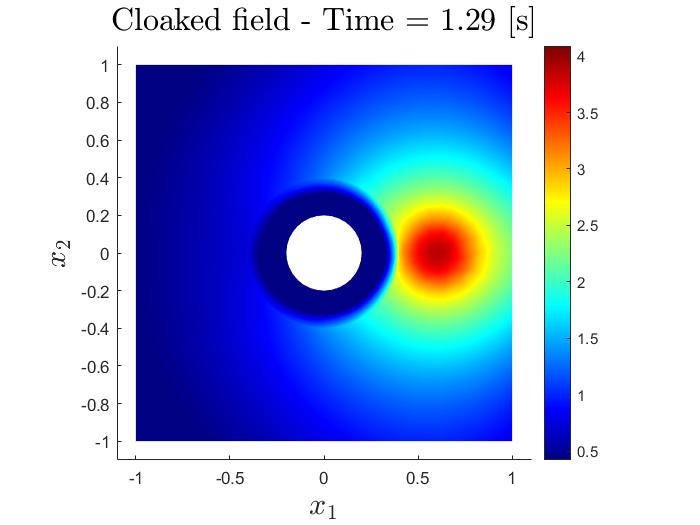}}
\centerline{
\includegraphics[width=0.365\textwidth]{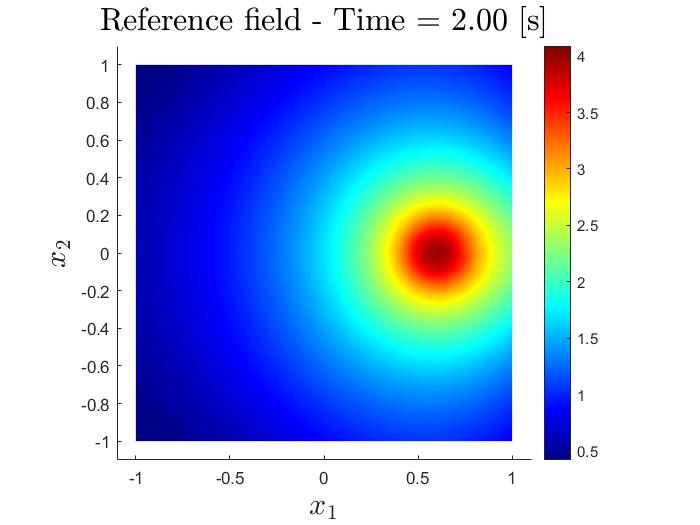}
\includegraphics[width=0.365\textwidth]{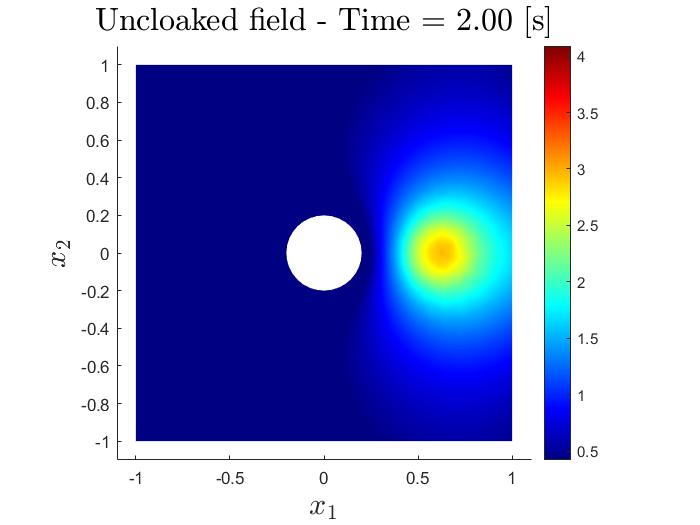}
\includegraphics[width=0.365\textwidth]{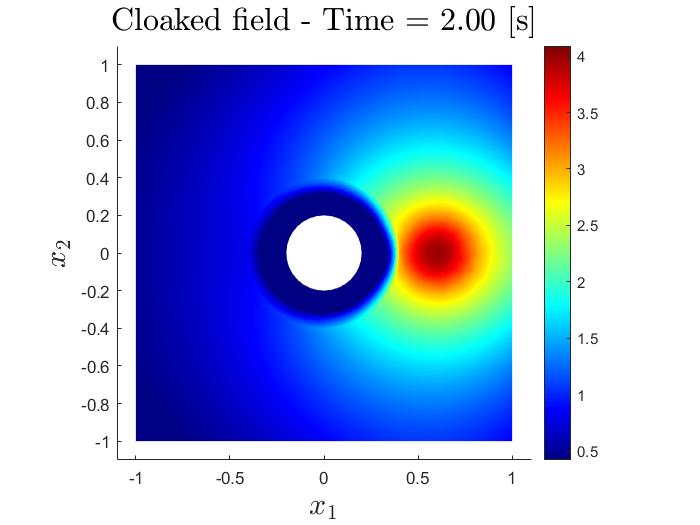}}
\caption{Time-dependent, passive thermal cloaking: reference, uncloaked, and cloaked fields (from left to right) at time instants $t=0.57(s)$, $t=1.29(s)$, $t=2.00(s)$ (from top to bottom).}
\label{state_time_dependent_solution}
\end{figure}

\begin{figure}[h!]
\centerline{
\includegraphics[width=0.365\textwidth]{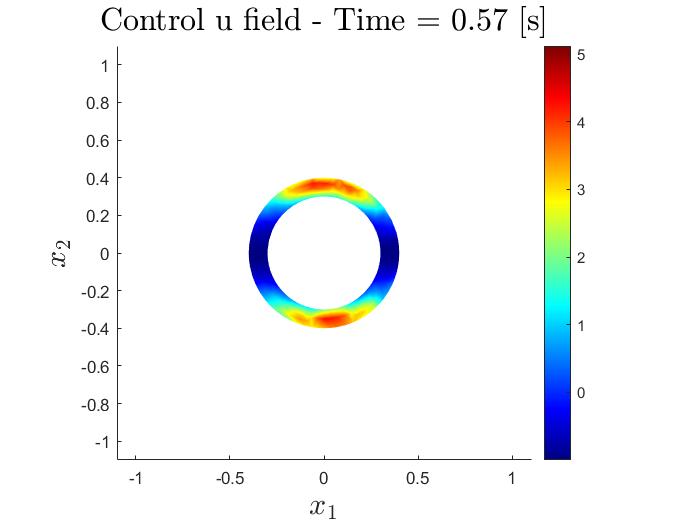}
\includegraphics[width=0.365\textwidth]{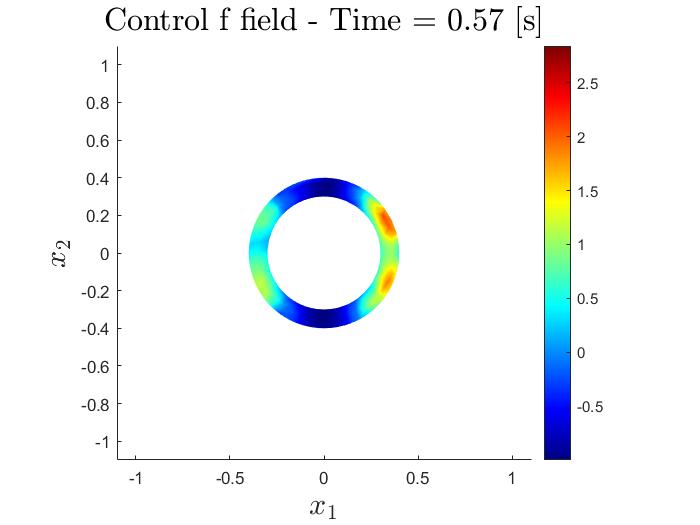}
\includegraphics[width=0.365\textwidth]{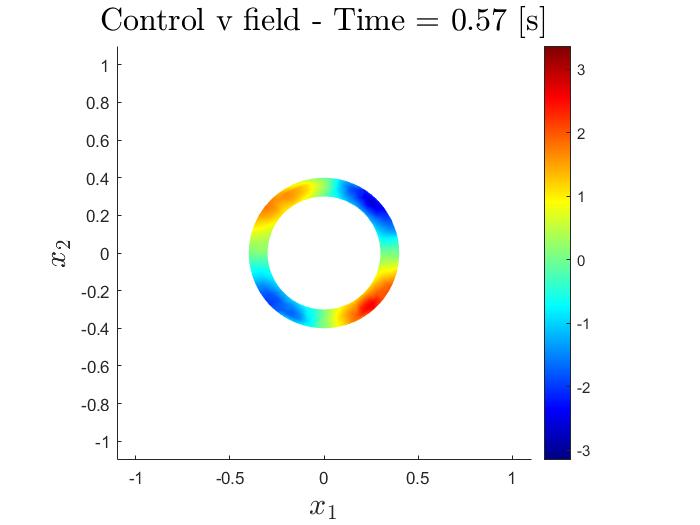}}
\centerline{
\includegraphics[width=0.365\textwidth]{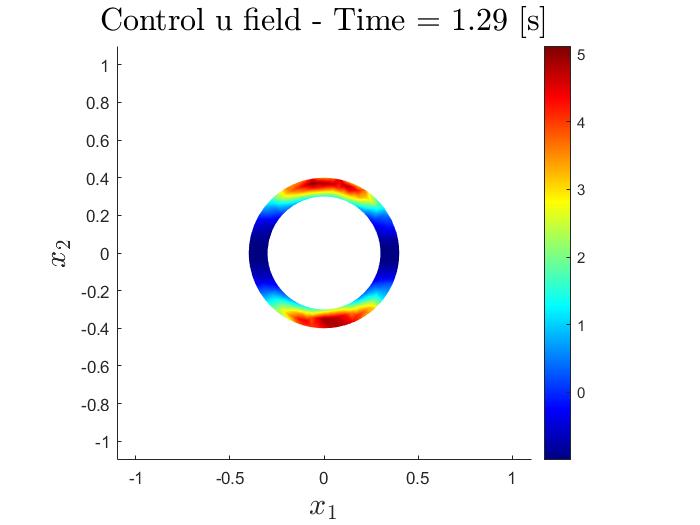}
\includegraphics[width=0.365\textwidth]{images/time_dependent_fmincon/control_f_0_57.jpg}
\includegraphics[width=0.365\textwidth]{images/time_dependent_fmincon/control_v_0_57.jpg}}
\centerline{
\includegraphics[width=0.365\textwidth]{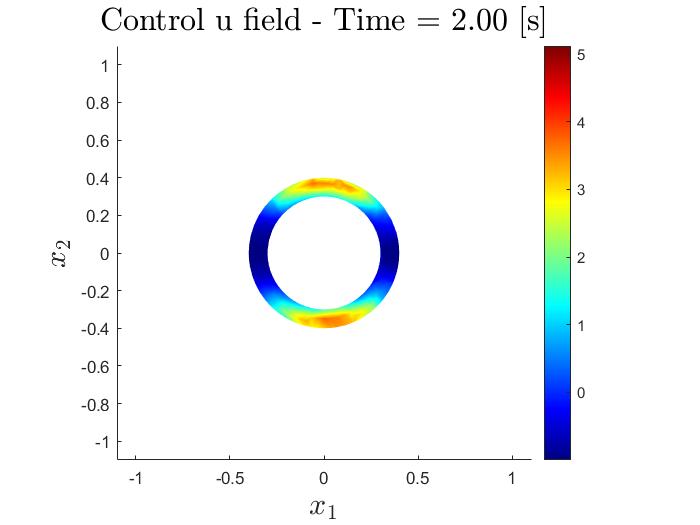}
\includegraphics[width=0.365\textwidth]{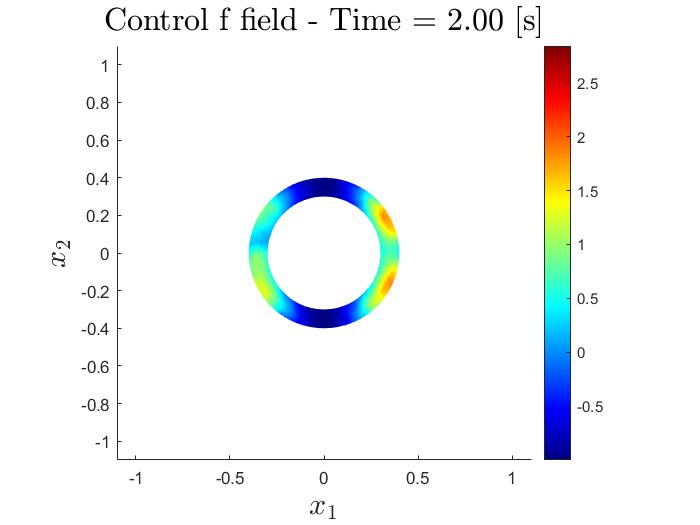}
\includegraphics[width=0.365\textwidth]{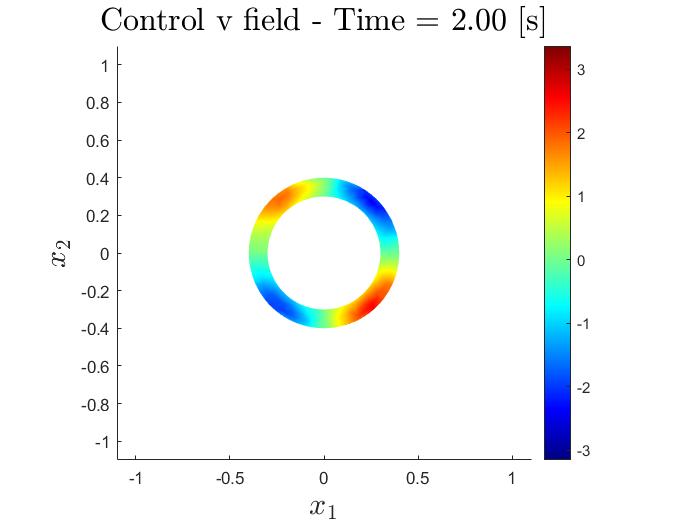}}
\caption{Time-dependent, passive thermal cloaking: 
control variables $u$, $f$, $v$ (lines 1--3) at time instants: $t=0.57(s)$, $t=1.29(s)$, $t=2.00(s)$ (from left to right). %\crl{anche qui si possono mostrare gli autovalori} \rcc{shown in the last plot}
}
\label{controls_time_dependent_case}
\end{figure}

\begin{figure}[h!]
\centerline{
\includegraphics[width=0.365\textwidth]{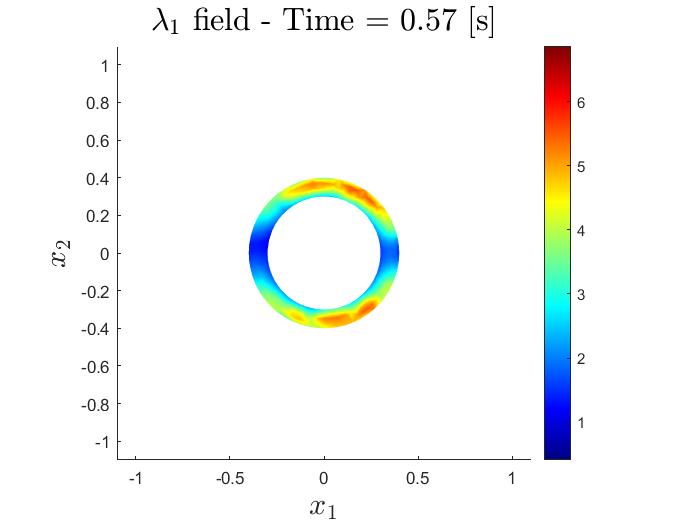}
\includegraphics[width=0.365\textwidth]{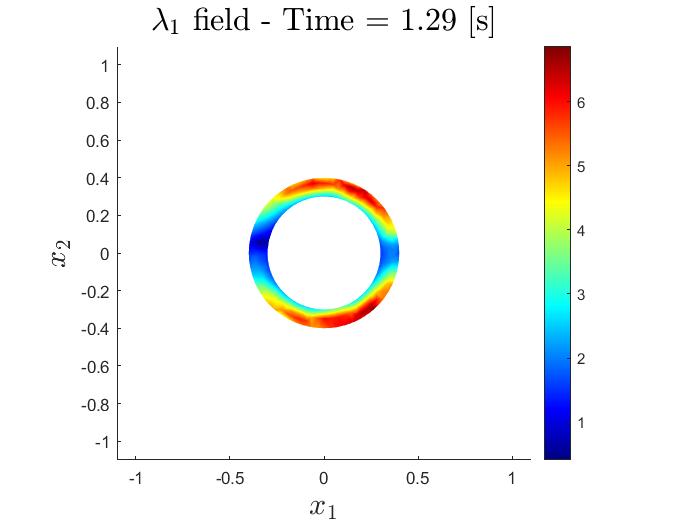}
\includegraphics[width=0.365\textwidth]{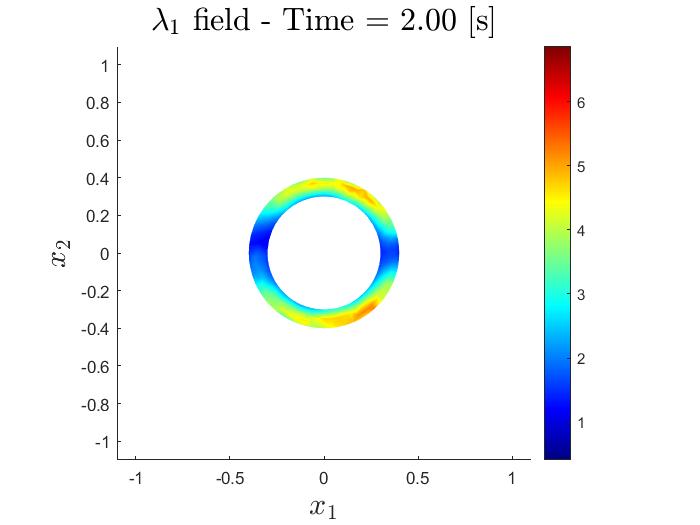}}

\caption{Time-dependent, passive thermal cloaking: 
eigenvalue $\lambda_1$ at time instants: $t=0.57(s)$, $t=1.29(s)$, $t=2.00(s)$ (from left to right). %\crl{anche qui si possono mostrare gli autovalori} \rcc{shown in the last plot}
}
\label{lambda_1_time_dependent_case}
\end{figure}

We finally evaluate the norm: $\|q-z\|_{L^2(0,T,L^2(\Omega_{obs}))}$ numerically approximating the time integral with the trapezoidal rule. The optimal solution gives us an approximated value $\|q^{\star}-z\|_{L^2(0,T,L^2(\Omega_{obs}))}^2\approx 0.2829$, while the temperature corresponding to the control variables initialized with unitary values gives $\|q-z\|_{L^2(0,T,L^2(\Omega_{obs}))}^2\approx 1.930$. Hence, the distance among reference and optimal state diminished of about $85.3\%$ throughout the optimization process, compared to the uncloaked configuration.

The cloaking efficiency is not constant in time. Moving towards the final time instant  $t=2 s$, we approach from above an efficiency of $\eta = 0.91$, that is close to the one reached in the stationary problem of Section \ref{Steady_state_numerical_simulation}. Lower values of the efficiency at the first time instants are due to a small overshot shown by the optimized field compared to the reference field. However, this discrepancy is quickly overcome and the tracking efficiency quickly increases in the following time instants. Moreover, at the first time instants, the effect of the heat flux produced by the source is so mild that the presence of the obstacle has little impact on the fields compared to the following time steps. In what follows, we compare the results obtained with a time dependent cloak with those one would obtain employing the cloak derived from the stationary optimization. In other words, the optimal controls of Figure \ref{ss_pb1} are applied constantly over time. In the first column of Figure \ref{state_stationary_time_dependent_solution} we show the state created by the stationary optimal control terms at times $t=0.57(s)$, $t=1.29(s)$, $t=2.00(s)$, the second column reports the tracking error produced by these very same control terms and the third column the tracking error produced by the optimal controls deriving from the time dependent optimization. 
\begin{figure}[h!]
\centerline{
\includegraphics[width=0.365\textwidth]{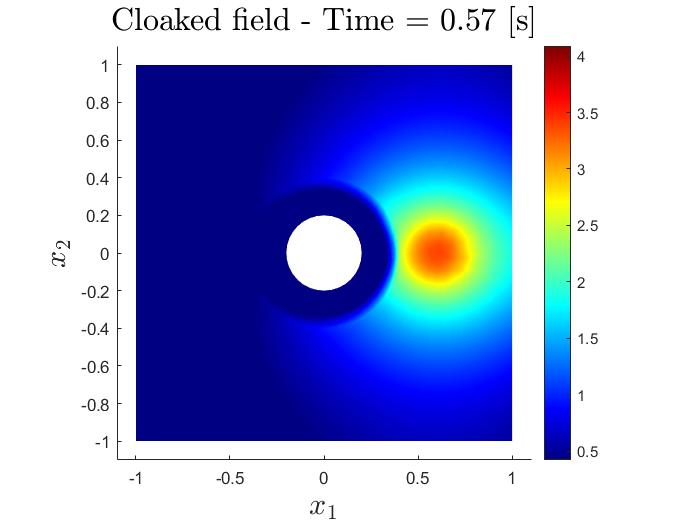}
\includegraphics[width=0.365\textwidth]{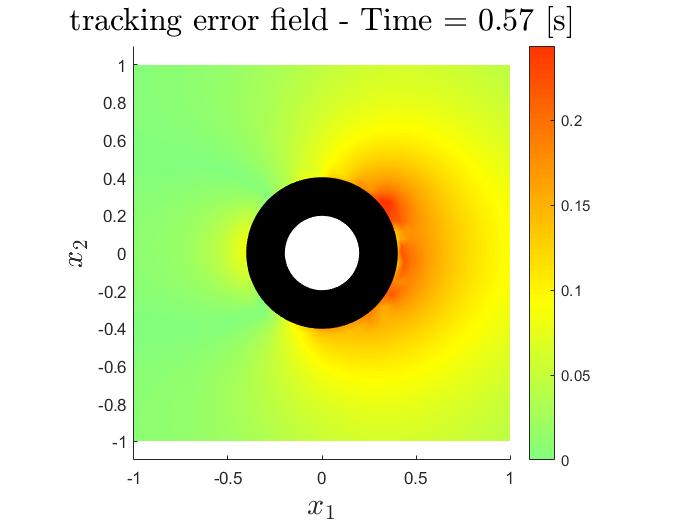}
\includegraphics[width=0.365\textwidth]{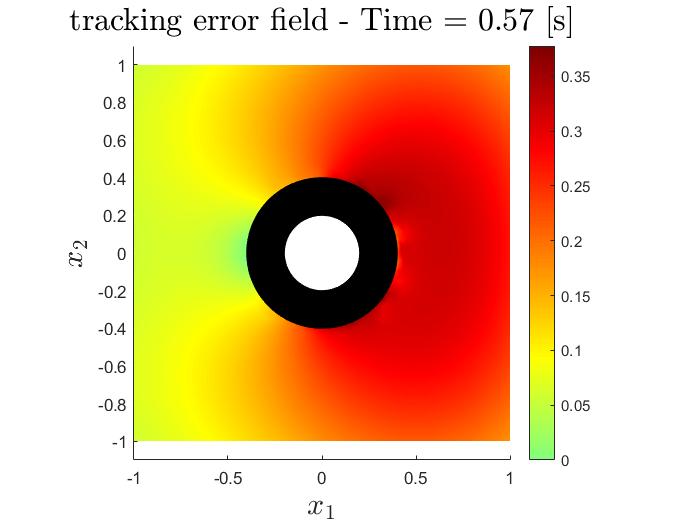}}
\centerline{
\includegraphics[width=0.365\textwidth]{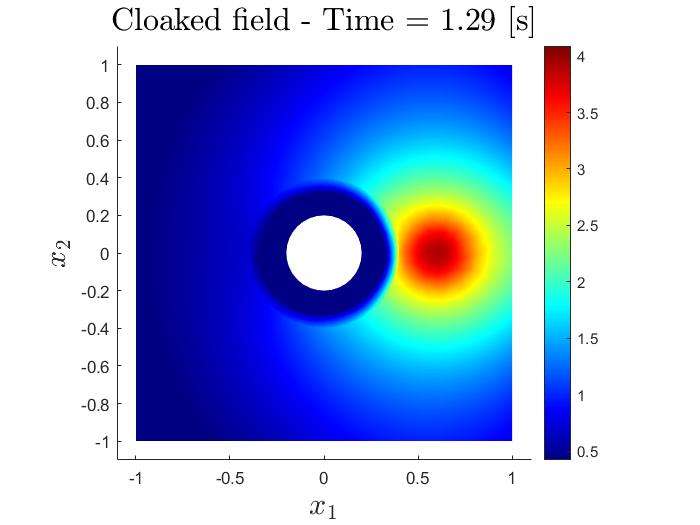}
\includegraphics[width=0.365\textwidth]{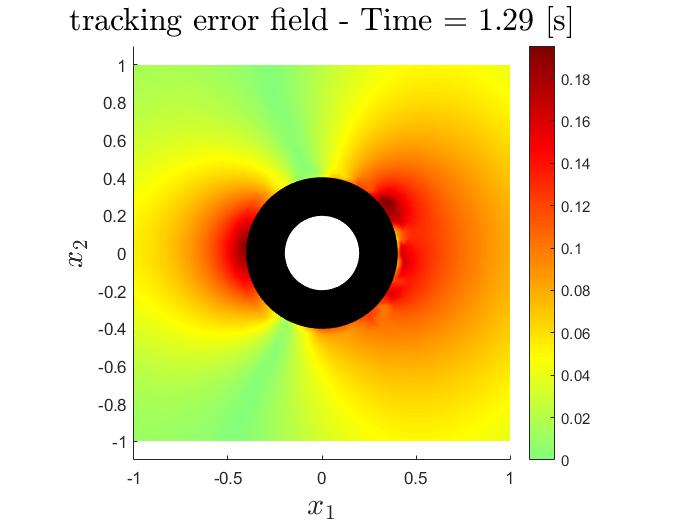}
\includegraphics[width=0.365\textwidth]{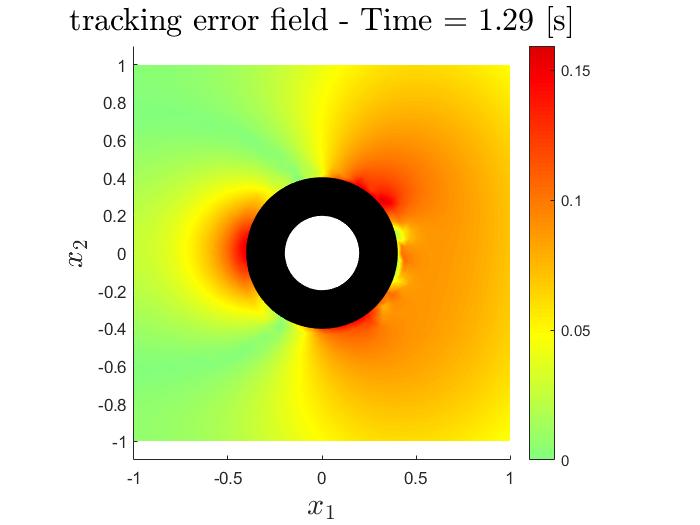}}
\centerline{
\includegraphics[width=0.365\textwidth]{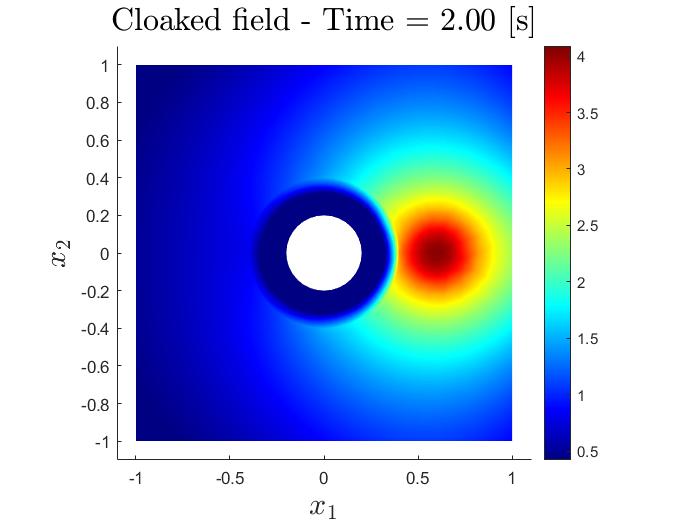}
\includegraphics[width=0.365\textwidth]{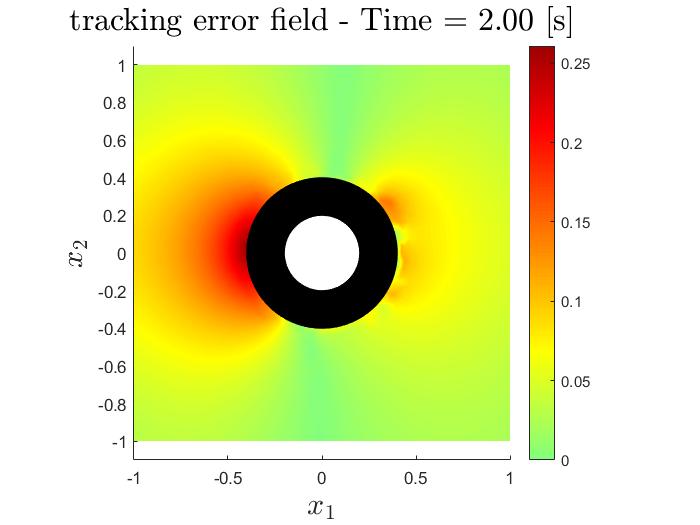}
\includegraphics[width=0.365\textwidth]{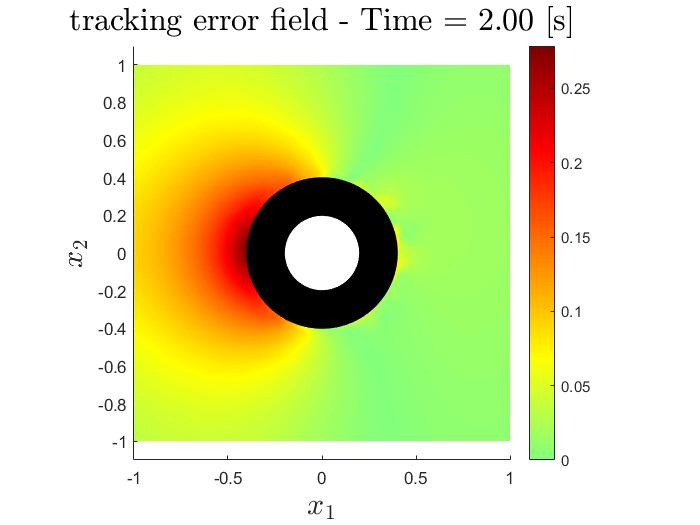}}

\caption{Time-dependent, passive thermal cloaking: cloaked field generated by the stationary optimal solution, tracking error field generated by the stationary optimal solution and tracking error generated by the time dependent optimal solution (from left to right) at time instants $t=0.57(s)$, $t=1.29(s)$, $t=2.00(s)$ (from top to bottom).}

\label{state_stationary_time_dependent_solution}
\end{figure}
From Figure \ref{state_stationary_time_dependent_solution}, it can be noticed that, at least at the first time instant, the stationary cloak outperforms the time dependent cloak. This is due to the inertia of the cloak: it requires some time to adjust its parameters to the incident field, while the stationary is already optimized on the final location of the source term.
However, we can notice that in the long run, at the end of the transitory phase, the optimal time dependent solution and the stationary solution tend to align. Indeed, the time dependent optimal controls are expected to converge to the steady-state ones. The reason why the passive cloak gives acceptable results at every time instant is that what influences mostly the performance of the cloak is the direction of the eigenvectors of the diffusivity matrix. This turns out to be mildly dependent on the source intensity and the source position. The performance of time-dependent arbitrarily shaped cloak is satisfactory as well. To show this, we approximate the cloaking problem in the case in which the control terms are placed in the silhouette of the half-woolen boar (Figure \ref{passive_boar_cloak}). The time-horizon is $t=2s$ and the time step is set to $\Delta{}t=0.5s$ ($N=4$). The reason for this increase in $\Delta{t}$ is that the meshing of the complex cloak is much more expensive. The other parameters are not varied. A sequence of solutions at time instants $t=0.50(s)$, $t=1.50(s)$, $t=2.00(s)$ is shown in Figure \ref{state_boar_time_dependent_solution}.  The control variables obtained at time instants $t=0.57(s)$, $t=1.29(s)$, $t=2.00(s)$ are shown in Figure \ref{control_boar_time_dependent_solution}.
\begin{figure}[h!]
\centerline{
\includegraphics[width=0.365\textwidth]{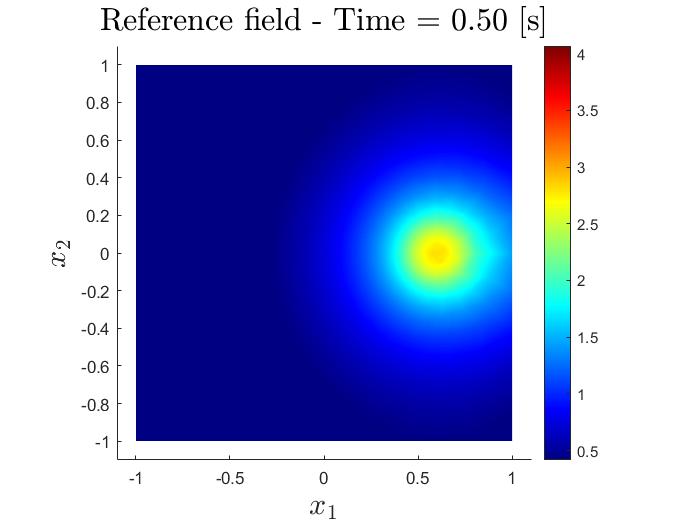  }
\includegraphics[width=0.365\textwidth]{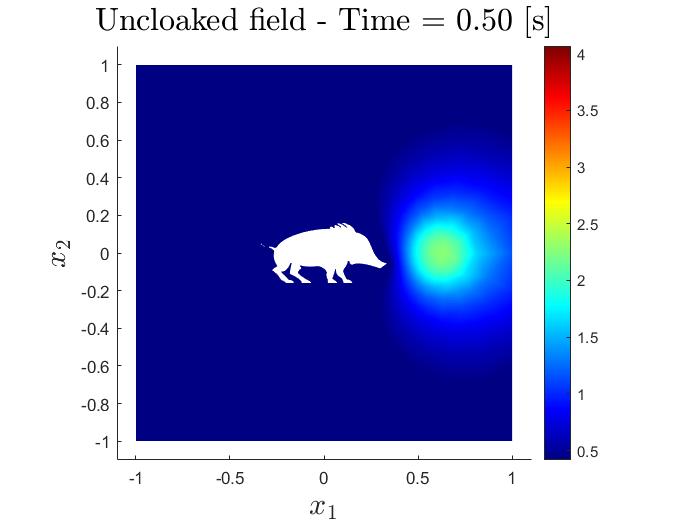}
\includegraphics[width=0.365\textwidth]{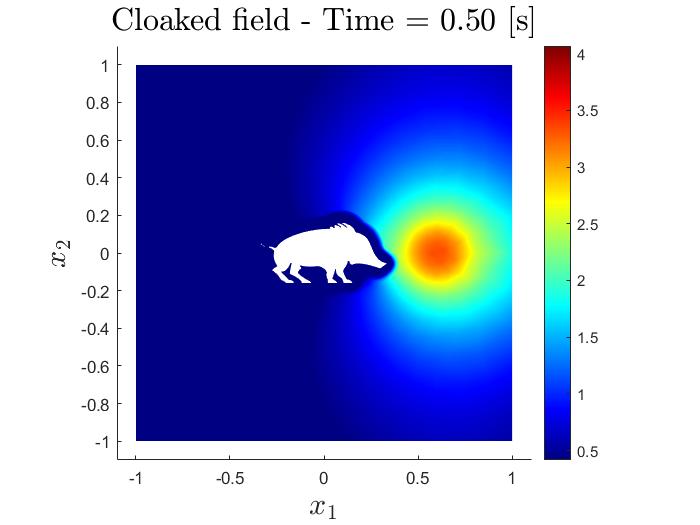}}
\centerline{
\includegraphics[width=0.365\textwidth]{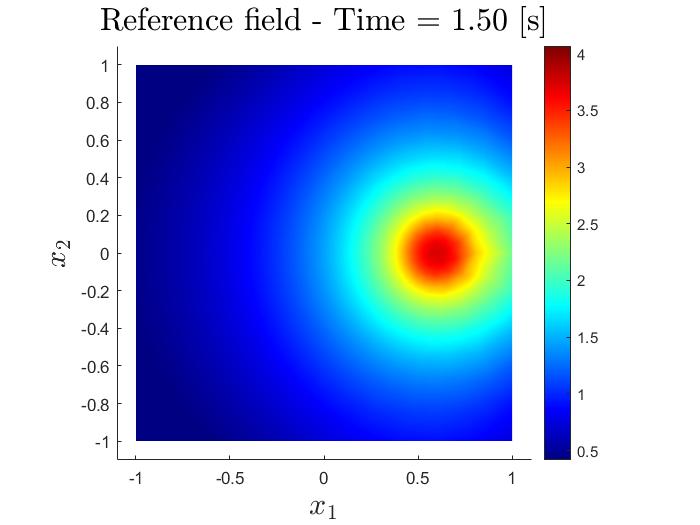}
\includegraphics[width=0.365\textwidth]{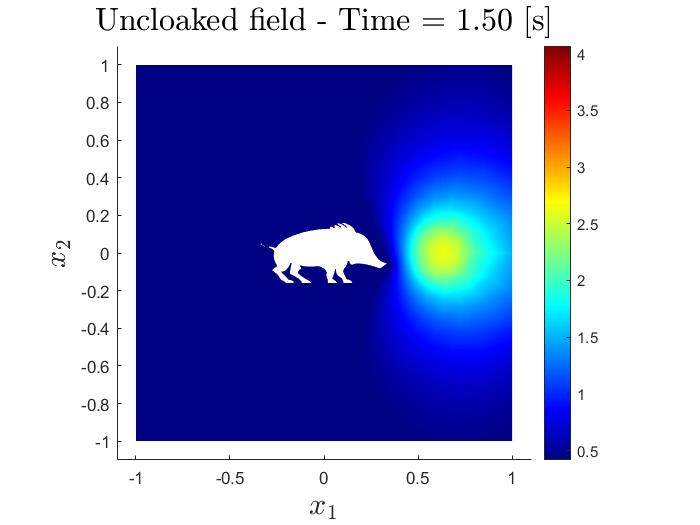}
\includegraphics[width=0.365\textwidth]{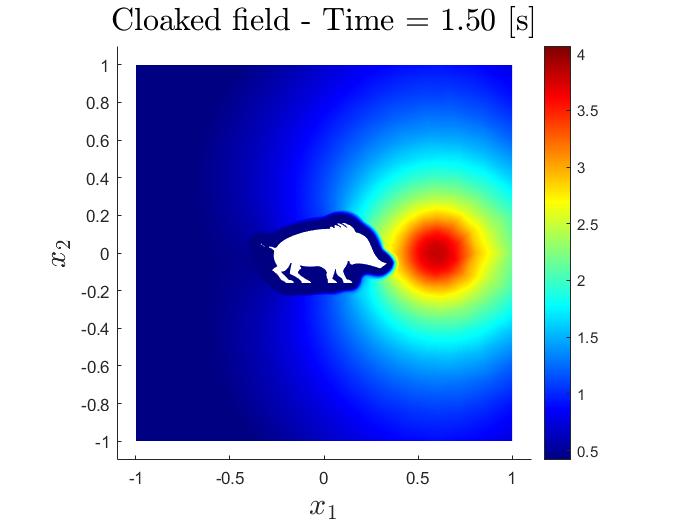}}
\centerline{
\includegraphics[width=0.365\textwidth]{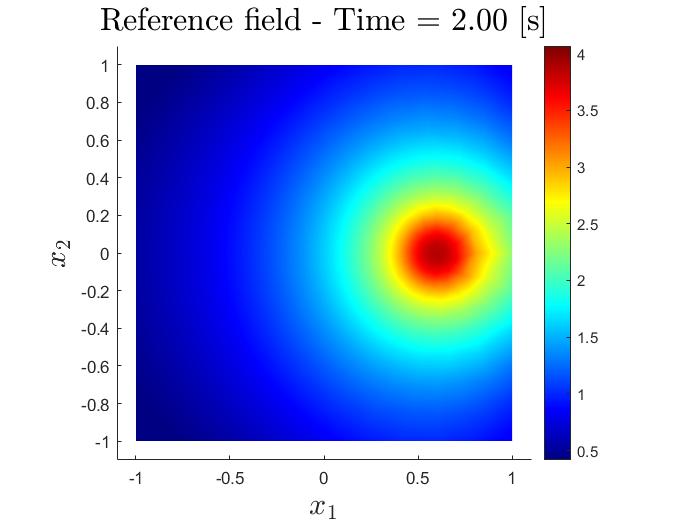}
\includegraphics[width=0.365\textwidth]{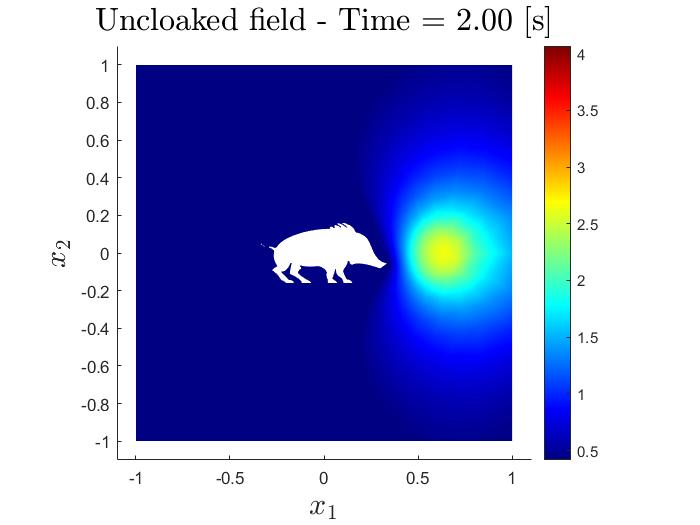}
\includegraphics[width=0.365\textwidth]{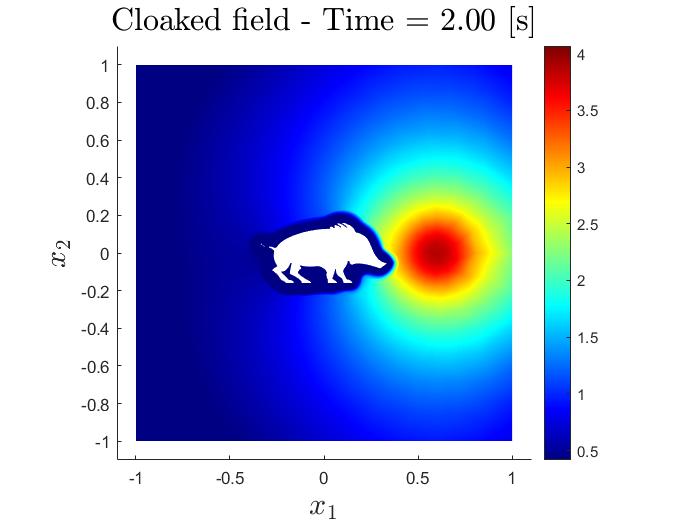}}

\caption{Time-dependent, passive thermal cloaking, target with complex shape: reference, uncloaked, and cloaked fields (from left to right) at time instants $t=0.50(s)$, $t=1.50(s)$, $t=2.00(s)$ (from top to bottom).}
\label{state_boar_time_dependent_solution}
\end{figure}
\begin{figure}[h!]
\centerline{
\includegraphics[width=0.365\textwidth]{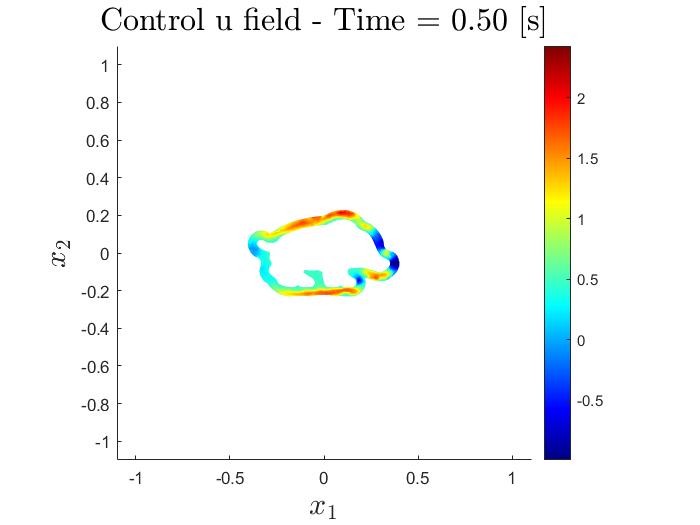}
\includegraphics[width=0.365\textwidth]{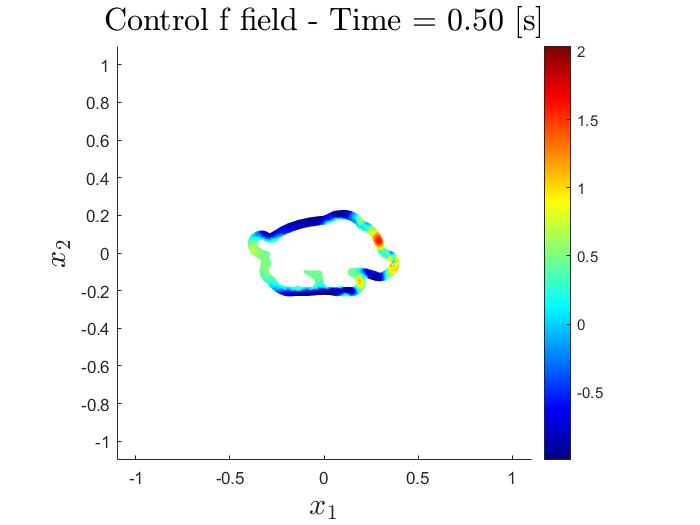}
\includegraphics[width=0.365\textwidth]{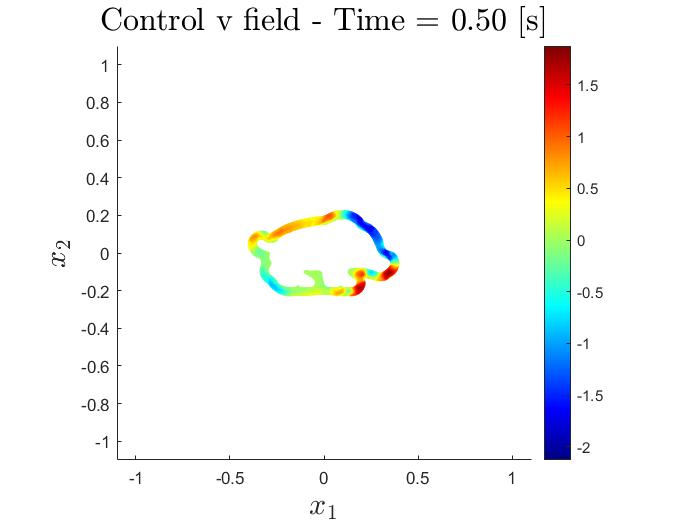}}
\centerline{
\includegraphics[width=0.365\textwidth]{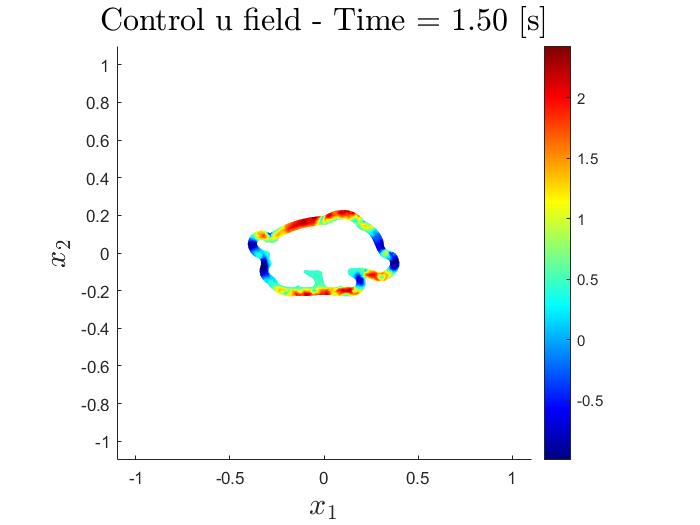}
\includegraphics[width=0.365\textwidth]{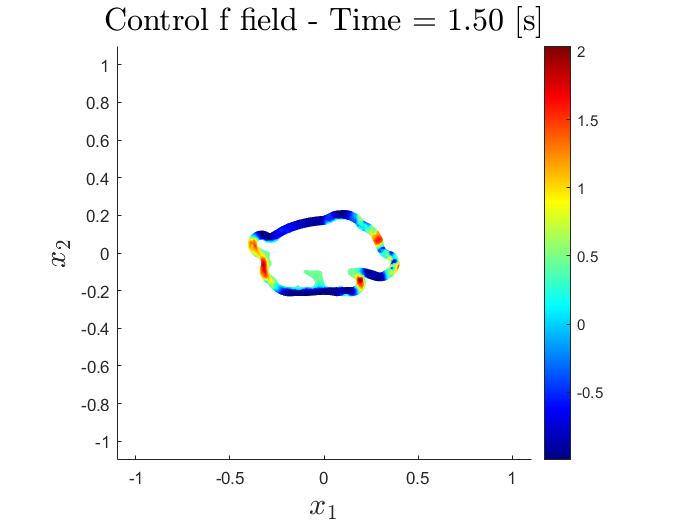}
\includegraphics[width=0.365\textwidth]{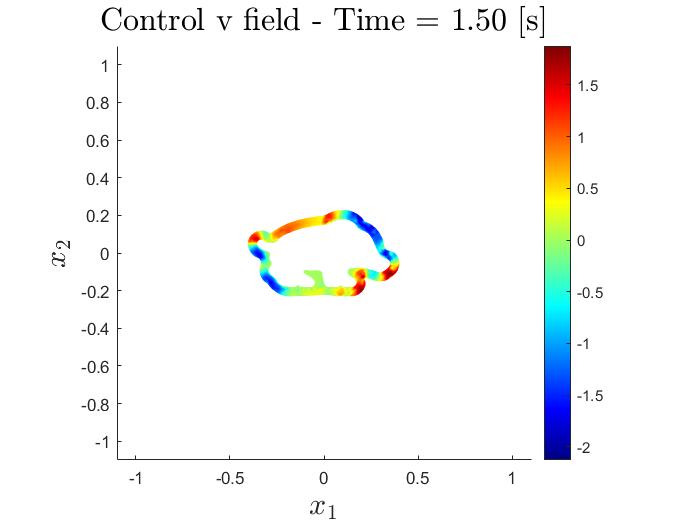}}
\centerline{
\includegraphics[width=0.365\textwidth]{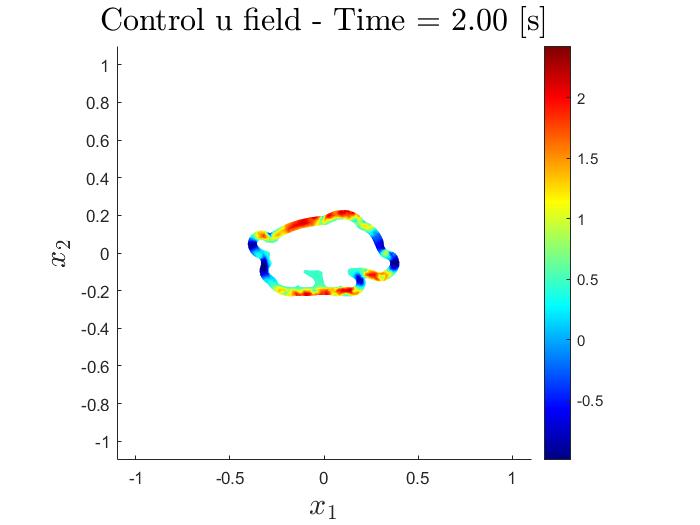}
\includegraphics[width=0.365\textwidth]{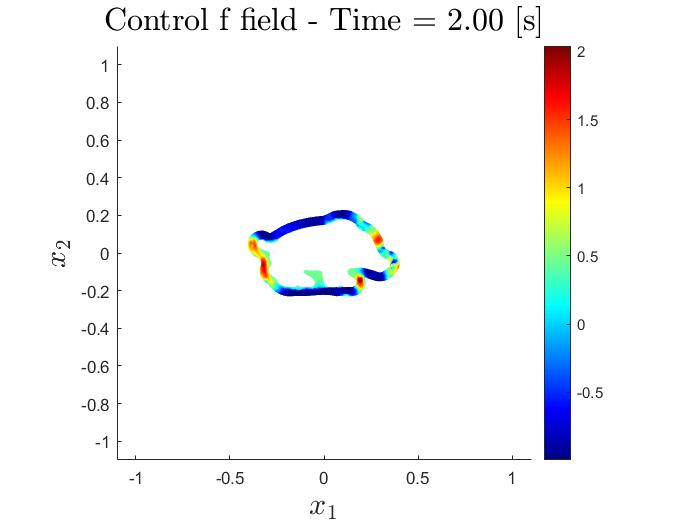}
\includegraphics[width=0.365\textwidth]{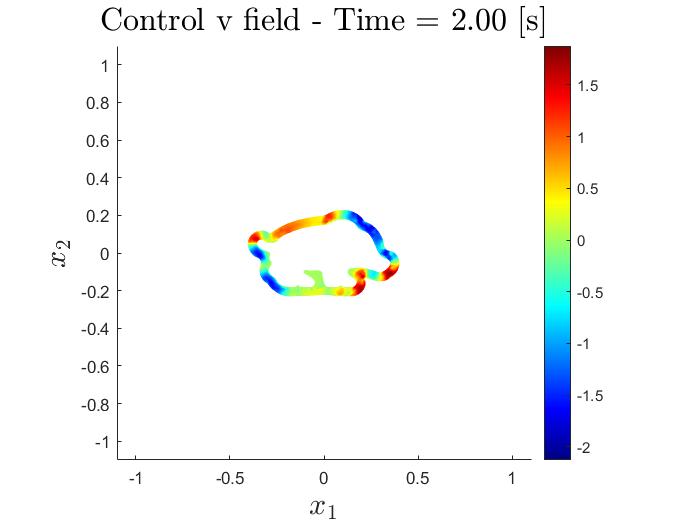}}
\caption{Time-dependent, passive thermal cloaking, target with complex shape:  $u$, $f$, $v$ (lines 1--3) at time instants: $t=0.50(s)$, $t=1.50(s)$, $t=2.00(s)$ (from left to right). %\crl{anche qui si possono mostrare gli autovalori} \rcc{shown in the last plot}
}
\label{control_boar_time_dependent_solution}
\end{figure}
To quantitatively assess the performance of the cloak, we report the value of the norm $\|\cdot\|_{L^2(0,T,L^2(\Omega_{obs}))}^2$. 
The optimal state gives $\|q^{\star}-z\|_{L^2(0,T,L^2(\Omega_{obs}))}^2\approx 0.1916$, while the field obtained with unitary control variable results in $\|q-z\|_{L^2(0,T,L^2(\Omega_{obs}))}^2\approx 2.3165$.
\\

%% file: sections/conclusion.tex
In this paper, we have designed a passive thermal cloaking strategy using optimization techniques. The arising optimal control problem is bilinear and this makes its well-posedness analysis nontrivial. In particular, an existence theorem for an optimal control and the associated optimal state has been proven, both in the stationary and the time-dependent case, before deriving a system of (necessary) optimality conditions. We have shown as the proposed optimal control strategy allows to impose a set of physical constraints regading heat conduction and convection during the optimization process. Constraints regarding the manufacturability of the cloak could be imposed as well, however requiring additional efforts and being out of the scope of the present work. 
The numerical approximation of the optimal control problems is derived following the Optimize-then-Discretize approach; results for both the stationary and time-dependent problem are reported. Compared to the previous work of Sinigaglia et al. {\cite{Sinigaglia_2022}, the passive cloaking device provides less flexibility. For instance, the performances deteriorate drastically when the cloaking region is a disconnected domain. However, as long as the cloaking sources are placed in a connected domain, we succeed in cloaking even obstacles of complex shapes maintaining the advantages offered by the passive configuration. Envisaging the possibility of practically realizing this cloak, further research will focus on the study of metamaterials presenting the inhomogeneous-anisotropic properties derived in the optimization. As future work, the setting of suitable reduced order modeling (ROM) strategies can be envisaged for the problems considered in this work. This would be invaluable to estimate the reference-field in real-time applications, and to promptly find the optimal solution. Indeed, the optimal solutions and control patterns are dependent on the particular probing field which needs to be known a-priori or estimated by measuring the temeperature field sorrounding the cloak. The optimization algorithm presented in this work would then contribute to construct an estimate, control and reduce algorithm able to tackle real-time applications and moving probing field. We remark that the ROM bears as additional difficulty the necessity to find a final solution fulfilling the desired constraints. This is far form being trivial, since imposing constraints on the reduced solution is not enough to guarantee that the recovered solution satisfies them. The efficient construction of accurate ROMs for constrained OCPs is still an open question, whose solution could impact on cloaking problems, among others, significantly. Finally, the time-depenent diffusivity field could serve as guideline to design active/robotic metamaterials see e.g. \cite{bertoldi2017flexible,brandenbourger2019non} which are able to change their conductivity properties dynamically and in a controlled way. Depending on the robotic material constraints, the design can be iterated exploiting the flexibility of the optimization paradigm in tackling practical actuator constraints.  \vspace{-0.48cm}

%% file: sections/appendix.tex
\section{Technical proofs}
\label{appendix_proof}
\subsection{Proof of Theorem \ref{Theorem_1} \label{proofTHM1}}
    Since $\mathcal{E}$ is uniformly elliptic, there exists $\hat{\alpha} > 0$ such that $K(\textbf{x})\pmb{\xi}\cdot\pmb{\xi} \geq \hat{\alpha}\pmb{\xi}\cdot\pmb{\xi}$ for all $\pmb{\xi}\in\mathbb{R}^{n},\: \forall \textbf{x}\in \Omega$. Therefore, the bilinear form $a(\cdot, \cdot)$ is coercive, since  \vspace{-0.2cm}
\begin{align*}
\begin{split}
    \textrm{$a(p,p)$} = &\int_{\Omega}(K{\nabla{}p})\cdot\nabla{}p \,d\Omega - \int_{{\Gamma_{d}}}p^2 \,d\Gamma = \int_{\Omega}(K{\nabla{}p})\cdot\nabla{}p \,d\Omega - \|p\|_{L^2(\Gamma_{d})}^2 \geq \\
    &\int_{\Omega}(\hat{\alpha}{\nabla{}p})\cdot\nabla{}p \,d\Omega - \|p\|_{L^2(\Gamma_{d})}^2 \geq \hat{\alpha} \|\nabla{}p\|_{L^2(\Omega)}^2 - C_{tr}^2 \|p\|_{H^1(\Omega)}^2 \geq %\\
    %& (\frac{\hat{\alpha}}{1+C_{p}^{2}} - C_{tr}^2) \|p\|_{H^1(\Omega)}^2 = 
    \tilde{\alpha}\|p\|^{2}_{V} \qquad \forall p \in V, \vspace{-0.4cm}
\end{split}
\end{align*}
having set $\tilde{\alpha} = (\frac{\hat{\alpha}}{1+C_{p}^{2}} - C_{tr}^2)$, thanks to the Trace and the Poincar\'e inequalities (see, e.g. \cite[Appendix A]{manzonioptimal}).
 A sufficient condition for coercivity is therefore that $\hat{\alpha} > ( 1+C_{p}^{2} )C_{tr}^2$.

Since $u\in\mathcal{U}\subseteq L^{\infty}(\Omega_{c})$ (see, e.g., \cite[Section A.5.11]{manzonioptimal}), then $K_{i,j}\in L^{\infty}(\Omega_{c})$, $i,j=1,2$, and therefore there exists $M>0$ such that 
$\|K\|_{L^{\infty}(\Omega_{c})} \leq M$. 
Therefore, the bilinear form is also continuous:
\begin{align*}
\begin{split}
    \textrm{$|a(u,v)|$} = &\lvert\int_{\Omega}(K\nabla{u})\cdot \nabla{v} \,d\Omega - \int_{\Gamma_{d}} uv \,d\Gamma\rvert \leq \int_{\Omega}|(K\nabla{u})\cdot \nabla{v}| \,d\Omega + \int_{\Gamma_{d}} |uv| \,d\Gamma \leq \\
&M\|\nabla{u}\|_{L^2(\Omega)}\|\nabla{v}\|_{L^2(\Omega)} + \|u\|_{L^2(\Gamma_{d})}\|v\|_{L^2(\Gamma_{d})} \leq ( M + C_{tr}^2 ) \|u\|_{H^1(\Omega)}\|v\|_{H^1(\Omega)}  \qquad \forall u,v \in V,
\end{split}
\end{align*}
where $\tilde{M}= M + C_{tr}^2>0$.
The functional $F$ is clearly continuous; then, applying the Lax-Milgram Lemma, there exists a unique $\tilde{q}\in V$ solving \eqref{weak_fomulation_analysis}; similarly, there exists a unique $q\in H^1(\Omega)$, $q = \tilde{q}+R T_0$, weak solution of \eqref{pde_analysis}. 
Then, taking $v=\tilde{q}$ in  the weak formulation \eqref{pde_analysis}, we obtain: 
\begin{equation*}
    a(\tilde{q},\tilde{q}) = \int_{\Omega} s\tilde{q} \,d\Omega - a(R T_0,\tilde{q}),
\end{equation*}
\vspace{-0.1cm}
and then
\begin{equation*}
\tilde{\alpha}\|\tilde{q}\|_{V}^2 \leq \int_{\Omega} s\tilde{q} \,d\Omega - a( R T_0 , \tilde{q} ) \leq |s|\|\tilde{q}\|_{L^2(\Omega)} + \tilde{M}\|R T_0\|_{V}\|\tilde{q}\|_{V} \leq |s|\|\tilde{q}\|_{V} + \tilde{M}\|R T_0\|_{V}\|\tilde{q}\|_{V}.
\end{equation*}
Finally, dividing by $\tilde{\alpha}\|\tilde{q}\|_{V}$ we get
$
    \|\tilde{q}\|_{V} \leq \frac{\lvert s\rvert + \tilde{M}\|R T_0\|_{V}}{\tilde{\alpha}}. \qedhere
$

\subsection{Proof of Theorem \ref{Theorem_2} \label{proofTHM2}}
    Taking   the derivative of the PDE \eqref{pde_analysis} with respect to $u$,  along the direction $h$, and using the chain rule with $q(u)=\Theta(u)$, we obtain that $z=\Theta'(u)h\in V$ satisfies
\begin{equation}  \label{directional_derivative_analysis}
 \begin{cases}
 \begin{array}{ll}
         \nabla{}\cdot[ - (\mu+u\chi_{c})\nabla{z}] = \nabla{}\cdot[h\chi_{c}\nabla{q}(u)]  & \text{in $\Omega$}  \\
        z = 0 & \text{on $\Gamma_{o}$}\\
        -\mu \nabla{z}\cdot \textbf{n} + z = 0 & \text{on $\Gamma_{d}$}.\\
    \end{array}
    \end{cases}
\end{equation}

The weak form of  problem \eqref{directional_derivative_analysis} therefore reads: find $z \in V$ such that 
\begin{equation} \label{weak_form_z_equation}
     \int_{\Omega} ( \mu + u\chi_{c} )\nabla{z}\cdot\nabla{\phi} \,d\Omega - \int_{\Gamma_{d}} z\phi \,d\Gamma = - \int_{\Omega} h\chi_{c}\nabla{q}\cdot\nabla{\phi} \,d\Omega \qquad \forall \phi \in V.
\end{equation}
Since \eqref{weak_form_z_equation} holds $\forall \phi \in V$, we can set $\phi = z$ and obtain:
\begin{equation*}
    a(z,z) = -\int_{\Omega_{c}} h\nabla{q}\cdot\nabla{z} \,d\Omega.
\end{equation*}

We can now use Hölder’s and Young inequality: $AB \leq \epsilon A^2 + \frac{1}{4\epsilon} B^2$ $\forall \epsilon > 0$ and the property of the norm $\| \cdot \|_{V}$ to rearrange the right hand side of the equation as:
\begin{equation*}
    \left|\int_{\Omega_{c}} h\nabla{q}\cdot\nabla{z} \,d\Omega \right| \leq \frac{1}{4\epsilon} \|h\nabla{q}\|^2_{L^2(\Omega_{c})} + \epsilon \|z\|_{V}^2 \qquad \forall \epsilon > 0 \vspace{-0.35cm}
\end{equation*}
where, in particular, we have used the embedding $H^2(\Omega_{c})\hookrightarrow L^{\infty}(\Omega_{c})$ for $n=2,3$ (see, e.g., \cite[Appendix A]{manzonioptimal}). 
Using the coercivity of the form $a(\cdot,\cdot)$ and rearranging the terms, we get:
\begin{equation} \label{estimate_norm_V_1}
     \|z\|_{V}^2 \leq \frac{1}{4\epsilon( \tilde{\alpha}-\epsilon )}\|h\nabla{q}\|_{L^2(\Omega_{c})}^2  \leq  \frac{C(n,p,\Omega_{c})}{4\epsilon( \tilde{\alpha}-\epsilon )}\|h\|^2_{L^{\infty}(\Omega_{c})}\|q\|_{V}^2 \qquad \forall \epsilon > 0;
\end{equation}
where $C(n,p,\Omega_{c})>0$ is the embedding constant, in general dependent on the domain $\Omega_{c}$, on the dimension $n$ and on the Hilbert space $H^p(\Omega_{c})$ considered.
Being $\epsilon>0$ arbitrary, we can pick $\epsilon = \frac{\tilde{\alpha}}{2}$ and so obtain from \eqref{estimate_norm_V_1}:
\begin{equation} \label{estimate_norm_V_2}
    \|z\|_{V}^2 \leq \frac{C(n,p,\Omega_{c})}{\tilde{\alpha}^2} \|h\|^2_{L^{\infty}(\Omega_{c})}\|q\|_{V}^2 \leq \frac{C(n,p,\Omega_{c})}{\tilde{\alpha}^2} \|h\|^2_{\mathcal{U}}\|q\|_{V}^2.
\end{equation}

We now consider the term of the Fr\'echet derivative: 
$R(u,h) = \Theta(u+h) - \Theta(u) - \Theta'(u)h$. In particular, 
$q(u+h)=\Theta(u+h)$ satisfies
\begin{equation}\label{weak_form_q_u+h}
    \int_{\Omega} (\mu + (u+h)\chi_{c})\nabla{q}(u+h)\cdot\nabla{p} \,d\Omega - \int_{\Gamma_{d}} q(u+h)p \,d\Gamma = Fp \quad \forall p\in V;
\end{equation}
then, subtracting to \eqref{weak_form_q_u+h} the weak formulation satisfied by $q(u)=\Theta(u)$ \eqref{weak_fomulation_analysis} and the weak formulation of $z= \Theta'(u)h$, given by \eqref{weak_form_z_equation},  we can infer that $R(u,h)\in V$ satisfies
\begin{align} \label{equation_R_analysis} 
       \int_{\Omega} [\mu + ( u + h )\chi_{c} ] \nabla{R}(u,h)\cdot \nabla{\phi}  \,d\Omega -\int_{\Gamma_{d}} R(u,h)\phi \,d\Gamma = -\int_{\Omega} h\chi_{c}\nabla{z}\cdot\nabla{\phi} \,d\Omega.
\end{align}

Since equation \eqref{equation_R_analysis} has the same form as \eqref{weak_form_z_equation}, we can immediately infer the following estimate for $R$:
\begin{equation} \label{bound_on_norm_V_R_1}
    \|R\|_{V}^2 \leq \frac{1}{\tilde{\alpha}^2}\|h\nabla{z}\|_{L^2(\Omega_{c})}^2 \qquad \forall \epsilon > 0  
\end{equation}
The term on the right hand side of \eqref{bound_on_norm_V_R_1} is controlled using \eqref{estimate_norm_V_1}:
\begin{equation*}
\|h\nabla{z}\|_{L^2(\Omega)}^2 \leq \|h\|_{L^{\infty}(\Omega_{c})}^2\|z\|_{V}^2\leq \frac{1}{\tilde{\alpha}^2} \|h\|_{L^{\infty}(\Omega_{c})}^4\|q\|_{V}^2 \leq \frac{C(n,p,\Omega_{c})}{\tilde{\alpha}^2} \|h\|_{H^2(\Omega_{c})}^4\|q\|_{V}^2
\end{equation*}
so that, by defining $C_{R}(\tilde{\alpha},n,p,\Omega_{c}) = \frac{C(n,p,\Omega_{c})}{\tilde{\alpha}^4}>0$, we obtain that 
\begin{equation} \label{bound_on_norm_V_R_2}
    \|R\|_{V}^2 \leq C_{R}(\tilde{\alpha},n,p,\Omega_{c}) \|h\|_{H^2(\Omega_{c})}^4\|q\|_{V}^2.
\end{equation}
Thanks to \eqref{lax-milgram bound}, the term $\|q\|_{V}^2$ can be bounded by the data. 

As a matter of fact, $R(u,h)=o(\|h\|_{\mathcal{U}})$ for $\|h\|_{\mathcal{U}}\rightarrow 0$ since
\begin{equation}
\lim_{\|h\|_{\mathcal{U}}\to 0} \frac{\|R(u,h)\|_{V}}{\|h\|_{\mathcal{U}}}  \leq \lim_{\|h\|_{\mathcal{U}}\to 0} C_{R}(\tilde{\alpha},n,p,\Omega_{c},\tilde{M},s,T_{o})\|h\|_{\mathcal{U}} = 0. 
\end{equation}
In conclusion, the control-to-state map satisfies: $q=\Theta(u)\in C^1(\mathcal{U},V)$. \qed

\subsection{Proof of Theorem \ref{Theorem_4} \label{proofTHM4}}
(Sketch)
Taking formally the derivative of our problem \eqref{pde_td_analysis} along the direction $u$ with respect to $h$ we note that $z=\Theta'(u)h\in Y$ satisfies \eqref{pde_z_td_analysis}.
Using \eqref{L2_V} it can be shown that following stability estimates for $z$ hold:
\begin{subnumcases}{}
     \|z\|_{L^2(0,T;V)}^2 \leq \tfrac{1}{\tilde{\alpha}^2}(\|h\|^2_{\mathcal{U}} K_{2}( s , q(0) , \tilde{\alpha}, T , n , p ,\Omega_{c}))\label{L2_V_z}\\
     \|z\|_{L^2(0,T;V^\star)}^2 \leq C_{1}(\tilde{M},\tilde{\alpha})(\|h\|_{\mathcal{U}}^2K_{2}( s , q(0) , \tilde{\alpha}, T ))\label{L2_V_z_star}
\end{subnumcases}
with $C_{1}(\tilde{M},\tilde{\alpha}) = \frac{1}{2}(\frac{\tilde{M}^2}{\tilde{\alpha}^2}+4)$. 
Analogously to what done in $theorem$ \ref{Theorem_2}, it can be shown that $R(u,h) = \Theta(u+h) - \Theta(u) - \Theta'(u)h$ satisfies
\begin{align} \label{R_weak} 
\begin{split}
        &\int_{\Omega} \frac{\partial{R}(u,h)}{\partial{t}}\phi \,d\Omega  + \int_{\Omega} ( \mu + ( u + h )\chi_{c} ) \nabla{R}(u,h)\cdot \nabla{\phi}  \,d\Omega -\int_{\Gamma_{d}} R(u,h)\phi \,d\Gamma = -\int_{\Omega} h\chi_{c}\nabla{z}\cdot\nabla{\phi} \,d\Omega.
\end{split}
\end{align}
Using \eqref{L2_V_z}, \eqref{L2_V_z_star} it turns out that $R$ satisfies
\begin{subnumcases}{}
     \|R(u,h)\|^2_{L^2(0,T;V)} \leq \frac{1}{\tilde{\alpha}^4} (\|h\|_{\mathcal{U}}^4K_{2}( s , q(0) , \tilde{\alpha}, T ))\label{L2_V_R}\\
     \|\dot{R}(u,h)\|_{L^2(0,T;V^\star)}^2 \leq (\frac{2\tilde{M}^2}{\tilde{\alpha}^4} + \frac{2}{\tilde{\alpha}^2})\|h\|_{\mathcal{U}}^4K_{2}( s , q(0) , \tilde{\alpha}, T , n , p ,\Omega_{c})\label{L2_R_V_dual},
\end{subnumcases}
and we pose $C_{2}(\tilde{M},\tilde{\alpha})=\frac{1}{\tilde{\alpha}^4}$, $C_{3}(\tilde{M},\tilde{\alpha}) = \frac{2(\tilde{M}^2+\tilde{\alpha}^2)}{\tilde{\alpha}^4}$.
Reminding that $\forall f\in H^1(0,T;V,V^{\star})$, we have
$
    \|f\|_{H^1(0,T;V,V^{\star})}^2 = \|f\|_{L^2(0,T;V)}^2 + \|f_{t}\|_{L^2(0,T;V^{\star})}^2.  
$
Using this definition for $R$, \eqref{L2_V_R} and \eqref{L2_R_V_dual}, we find
\begin{equation*} 
    \|R(u,h)\|_{Y} \leq \sqrt{[(C_{2}(\tilde{\alpha})+C_{4}(\tilde{M},\tilde{\alpha}))]K_{2}( s , q(0) , \tilde{\alpha}, T , n , p ,\Omega_{c})}\|h\|_{\mathcal{U}}^2,
\end{equation*}
telling us that $R(u,h)=o(\|h\|_{\mathcal{U}})\:for\:\|h\|_{\mathcal{U}}\rightarrow 0$.
Therefore the control-to-state map satisfies: $q=\Theta(u)\in C^1(\mathcal{U},Y)$. \qed